\documentclass[a4paper,12pt]{article}
\usepackage{amsmath}
\usepackage{amssymb}
\usepackage{amsthm}	

\usepackage{latexsym}
\usepackage[dvipdfmx]{graphicx}
\usepackage{txfonts}
\usepackage{bm}
\usepackage{color}
\usepackage{epic,eepic}

\usepackage{geometry}
\numberwithin{equation}{section}

\newtheorem{theorem}{Theorem}[section]
\newtheorem{proposition}[theorem]{Proposition}
\newtheorem{corollary}[theorem]{Corollary}
\newtheorem{lemma}[theorem]{Lemma}
\newtheorem{remark}{Remark}[section]
\newtheorem{example}{Example}[section]

\newcommand{\todaye}{\the\year/\the\month/\the\day}
\newcommand{\finbox}{\hspace*{\fill}$\rule{0.2cm}{0.2cm}$}

\newcommand{\RR}{{\mathbb{R}}}
\newcommand{\ZZ}{{\mathbb{Z}}}
\newcommand{\vecone}{{\bm{1}}}
\newcommand{\veczero}{{\bm{0}}}
\newcommand{\dom}{{\rm dom\,}}

\newcommand{\suppp}{{\rm supp}^{+}}
\newcommand{\suppm}{{\rm supp}^{-}}

\newcommand{\argmin}{\arg \min}

\newcommand{\OMIT}[1]{{\bf [OMIT:} #1 \ {\bf --- end OMIT] }}  %%% For work
   \renewcommand{\OMIT}[1]{}            %%% For FINAL

\title{
Discrete Midpoint Convexity
}%title%%%%%%%%%

\author{
Satoko Moriguchi%
\thanks{%
Department of Economics and Business Administration,
Tokyo Metropolitan University, 
satoko5@tmu.ac.jp},
Kazuo Murota%
\thanks{%%
Department of Economics and Business Administration,
Tokyo Metropolitan University, 
murota@tmu.ac.jp}, 
Akihisa Tamura%
\thanks{Department of Mathematics, Keio University, 
aki-tamura@math.keio.ac.jp},
Fabio Tardella%
\thanks{Department of Methods and Models for Economics, 
Territory and Finance, 
Sapienza University of Rome, 
fabio.tardella@uniroma1.it}
}%%author

\date{August 2017; November 2018}
\begin{document}

\maketitle

%%\tableofcontents

\begin{abstract}
For a function defined on the integer lattice, we consider 
discrete versions of midpoint convexity,
which offer a unifying framework for discrete convexity of functions,
including integral convexity,  ${\rm L}^{\natural}$-convexity, and  submodularity.
By considering discrete midpoint convexity for all pairs at $\ell_\infty$-distance
equal to two or not smaller than two,
we identify new classes of discrete convex functions,
called locally and globally discrete midpoint convex functions.
These functions enjoy nice structural properties.
They are stable under scaling and addition,
and satisfy a family of inequalities named parallelogram inequalities.
Furthermore, they admit a proximity theorem with the same small proximity bound as
that for ${\rm L}^{\natural}$-convex functions.
These structural properties allow us to develop an algorithm 
for the minimization of 
locally and globally discrete midpoint convex functions based on the proximity-scaling
approach and on a novel 2-neighborhood steepest descent algorithm.
\end{abstract}

%%\subclass{52A41 \and 90C27 \and 90C25}
%% 52B40 = matroid
%% 52A41=Convex functions and convex programs
%% 90C10=Integer programming
%% 90C27 = combinatorial optimization
%% 90C25 = convex programming

%%%%%%%%%%%%%% SSSSS %%%%%%%%%%%%%%%%%%%%%
\section{Introduction}
\label{SCintro}

For a function $f$ defined on a convex set $S \subseteq \RR\sp{n}$, \emph{midpoint convexity} is
the property requiring that
\begin{equation} \label{Realmidconv}
 f(x) + f(y) \geq
    2f \left(\frac{x+y}{2} \right)
\end{equation}
for all $x,y \in S$.
It may be worthwhile 
to recall that the first definition and systematic study of convex functions appeared
more than a century ago in a milestone paper by Jensen \cite{Jen0506}, where midpoint
convexity was used to define convex functions. The now classical defining inequality
for convex functions
\begin{equation} \label{Realconv}
 \alpha f(x) + (1 - \alpha) f(y) \geq
    f(\alpha x + (1 - \alpha) y)
\qquad (0 \leq \alpha \leq 1)
\end{equation}
was then proved by Jensen to be implied by midpoint convexity 
in the case of continuous functions and of functions bounded above. 
In the same paper, the celebrated Jensen inequality was also derived from
midpoint convexity.  The equivalence between midpoint convexity (\ref{Realmidconv}) and classical convexity
(\ref{Realconv}) was later shown to hold under very mild assumptions such as measurability or local boundedness
above (see, e.g., \cite{Bec48}, \cite{Tho94} and references therein).
It is apparent that midpoint convexity implies nonnegativity
of the second order directional differences
$D\sp{2}_h f(x) = f(x+h)-2f(x) + f(x-h)$
for all $x,h \in \RR^{n}$ such that $x+h, x, x-h \in S$.
In the case of twice continuously differentiable functions this clearly implies
that the second order derivatives of $f$ are nonnegative in all directions $h$, and hence that
the Hessian of $f$ is positive semidefinite.

For a function defined on the integer lattice $\ZZ^{n}$, we may consider analogous notions
of \emph{discrete} midpoint convexity where the value of the function at the 
(possibly nonintegral) midpoint
 is replaced by suitable convex combinations of the function values at some integer neighboring points.

For $x,y \in \ZZ^{n}$, one very simple substitute for
$f((x+y)/2)$
in the case of noninteger
$(x+y)/2$ 
is obtained by taking the average of the values of $f$
at the integer round-up and round-down of the point $(x+y)/2$.
In this case we obtain the 
inequality of \emph{discrete midpoint convexity}:
\begin{equation} \label{discmptconvfn}
 f(x) + f(y) \geq
    f \left(\left\lceil \frac{x+y}{2} \right\rceil\right)
  + f \left(\left\lfloor \frac{x+y}{2} \right\rfloor\right),
\end{equation}
where $\lceil \cdot \rceil$ and $\lfloor \cdot \rfloor$ denote
the integer vectors obtained by
rounding up and rounding down each component to the nearest integers,
respectively.

A weaker version of discrete midpoint convexity is obtained 
by considering the smallest possible value of a linear extension 
of $f$ at the point $(x+y)/2$
with respect to the values 
at the integer points neighboring $(x+y)/2$.
More precisely, for a point $x \in \RR^n$ we consider its integer neighborhood
$N(x) = \{ z \in \mathbb{Z}^{n} \mid | x_{i} - z_{i} | < 1 \ (i=1,\ldots,n)  \}$
and the value at $x$ of the convex envelope
$\tilde f$ of $f$ on $N(x)$, which can be represented as
\begin{equation} \label{fnconvclosureloc0}
 \tilde f(x) =
  \min\{ \sum_{z \in N(x)} \lambda_{z} f(z) \mid
      \sum_{z \in N(x)} \lambda_{z} z = x, \
  (\lambda_{z})  \in \Lambda(x) \}
\quad (x \in \RR^{n}) ,
\end{equation}
where
$\Lambda(x)$
denotes the set of coefficients
$(\lambda_{z} \mid z \in N(x) ) \in \RR\sp{N(x)}$
for convex combinations indexed by $N(x)$.
We then consider the following 
inequality of \emph{weak discrete midpoint convexity}:
\begin{equation} \label{weakmidconv}
 f(x) + f(y) \geq  2 \tilde f \left(\frac{x+y}{2} \right) ,
\end{equation}
which is indeed a weaker condition than (\ref{discmptconvfn}) since
$\frac{x+y}{2} = \frac{1}{2} \left\lceil \frac{x+y}{2} \right\rceil +
\frac{1}{2} \left\lfloor \frac{x+y}{2} \right\rfloor  $
implies
$2  \tilde f \left(\frac{x+y}{2} \right)
\leq f \left(\left\lceil \frac{x+y}{2} \right\rceil \right)
  + f \left(\left\lfloor \frac{x+y}{2} \right\rfloor \right)$
by (\ref{fnconvclosureloc0}).

The objective of this paper is twofold:
(i) to highlight discrete midpoint convexity as 
a unifying framework for convexity concepts of functions on the integer lattice $\ZZ^{n}$,
and (ii) to investigate structural and algorithmic properties
of functions defined by versions of discrete midpoint convexity.

New and known classes of functions can be obtained by
requiring discrete midpoint convexity (\ref{discmptconvfn}) or 
weak discrete midpoint convexity (\ref{weakmidconv}) for all points $x,y$
at a prescribed $\ell_\infty$-distance.
It is known that
weak discrete midpoint convexity (\ref{weakmidconv})
coincides with integral convexity \cite{FT90}, 
and that discrete midpoint convexity (\ref{discmptconvfn}) for all $x,y$
characterizes ${\rm L}^{\natural}$-convexity \cite{FM00,Mdcasiam},
whereas (\ref{discmptconvfn}) 
for all $x,y$ at $\ell_\infty$-distance one
 characterizes submodularity \cite{Fuj05book} (see Section \ref{SC-DMC} for details).
By requiring discrete midpoint convexity (\ref{discmptconvfn})
 for all pairs $(x,y)$ at $\ell_\infty$-distance
equal to two or not smaller than two,
we can define new classes of discrete convex functions,
called locally and globally discrete midpoint convex functions,
which are strictly between the classes of ${\rm L}^{\natural}$-convex
and integrally convex functions.

Locally and globally discrete midpoint convex functions
enjoy nice structural properties
possessed by ${\rm L}^{\natural}$-convex functions and not by general integrally convex functions.
Discrete midpoint convex functions
are stable under addition and scaling (Proposition \ref{PRdirintcnvinvarS}, Theorem \ref{THdirintsc}),
and satisfy a family of inequalities 
which are named {\em parallelogram inequalities} (Theorem \ref{THparallelineqgen}).
Furthermore, they admit a proximity theorem
with the same small proximity bound as that for ${\rm L}^{\natural}$-convex functions (Theorem \ref{THdirintprox}).
These structural properties allow us to develop a proximity-scaling based algorithm 
for the minimization of locally and globally discrete midpoint convex functions.

Algorithms based on scaling and proximity   
have been successful for discrete optimization problems such as
resource allocation problems \cite{Hoc07,HS90,IK88,KSI13}
and convex network flow problems
(under the name of ``capacity scaling'') 
\cite{AMO93,IMM05submflow,IS03}.
For separable convex functions defined on discrete rectangles (boxes),
the proximity-scaling approach is straightforward.
${\rm L}^{\natural}$-convex functions are also amenable to this approach:
the scaling operation preserves ${\rm L}^{\natural}$-convexity,
and the proximity theorem holds.
Efficient algorithms for minimizing ${\rm L}^{\natural}$-convex functions 
have been successfully designed with the proximity-scaling approach
\cite{MT09Lrelax,Mdcasiam} and also for {\rm L}-convex functions on graphs 
\cite{Hir15Lextprox,Hir16Lgraph}.
Recently a proximity-scaling algorithm has been 
developed for integrally convex functions in \cite{MMTT17proxIC}. 
The algorithm proposed in this paper employs 
the standard proximity-scaling framework
while using a novel variant of descent algorithm suited for discrete midpoint convex functions.
Our descent algorithm,
named the {\em 2-neighborhood steepest descent algorithm},
repeats finding a local minimizer in the neighborhood of $\ell_{\infty}$-distance 2,
and the number of iterations is shown to be exactly equal to 
the minimum $\ell_{\infty}$-distance to a minimizer,
which extends the results of \cite{KS09lnatmin,MS14exbndLmin}
for ${\rm L}^{\natural}$-convex functions. 
Our scaling algorithm finds
a minimizer of a discrete midpoint convex function
with ${\rm O}(5^{n} n \log_{2} K_{\infty})$
function evaluations, 
where $K_{\infty}$ denotes the $\ell_{\infty}$-size of the effective domain of the function.
This means that, if the dimension $n$ is fixed,
the algorithm is polynomial in the problem size.
Besides the proximity approach,
there are many other approaches to nonlinear integer optimization
\cite{LL12book}.
Among others, methods based on sophisticated algebraic tools
have been developed in the last two decades, as described in 
\cite{DLHK13book,HKLW10,Onn10book} as well as \cite{LL12book}.

In the next section we describe in detail the relations of discrete
midpoint convexity with integral convexity, ${\rm L}^{\natural}$-convexity,
and submodularity. Furthermore, we introduce the new classes of locally
and globally discrete midpoint convex sets and functions, and we analyze their basic
properties. For locally and globally discrete midpoint convex functions
we establish a useful ``parallelogram inequality'' in Section~\ref{SCdirintparaineq},
while in Sections \ref{SCscalingDIC} and \ref{SCdirintcnvprox} we 
present scaling and proximity results.
Such results are then used to develop a 
scaling algorithm for minimization
in Section~\ref{SCminalg}.
Quadratic functions with discrete midpoint convexity are considered in 
Section~\ref{SCquadfnDIC}, while Section~\ref{SCproofgen}
is devoted to technical proofs of basic facts about discrete midpoint convex sets.

\section{Discrete Midpoint Convexity}
\label{SC-DMC}

In this section we first provide an overview of the relations between functions
satisfying (weak) discrete midpoint convexity for all pairs of points 
at a prescribed $\ell_\infty$-distance and the classes of submodular,
integrally convex, and ${\rm L}^{\natural}$-convex functions in the simpler case where
the functions are defined
on a finite rectangular domain (box) of $\ZZ^{n}$. 
Then we show that the same relations also hold in the case 
of more general domains. Furthermore, we show that 
there exist only two classes of discrete midpoint 
convex functions that do not coincide with previously known 
classes of functions. We call these new classes locally and globally 
discrete midpoint convex functions and we describe here their basic properties.

\subsection{New and known classes of discrete midpoint convex functions}
\label{SCnewandoldDMC}

We introduce some notations to better illustrate various classes of discrete midpoint convex
functions and their relations with other known classes of functions.
We denote by DMC($k$)
and by DMC({$\geq$}{$k$})  the classes of functions
$f: \ZZ^{n} \to \RR \cup \{ +\infty \}$
that satisfy discrete midpoint convexity (\ref{discmptconvfn}) for all
$x,y \in \ZZ^n$ with $\|x-y \, \|_\infty = k$
and for all
$x,y \in \ZZ^n$ with $\|x-y \, \|_\infty \geq k$,
respectively.
Similarly, we denote by WDMC($k$) and by WDMC({$\geq$}{$k$})  the classes
of functions that satisfy weak discrete midpoint convexity (\ref{weakmidconv}) for all
$x,y \in \ZZ^n$ with $\|x-y \, \|_\infty = k$ and for all
$x,y \in \ZZ^n$ with $\|x-y \, \|_\infty \geq k$, respectively.

We recall that a function
$f: \ZZ^{n} \to \RR \cup \{ +\infty \}$
is \emph{submodular} if
\begin{equation} \label{submfndef}
 f(x) + f(y) \geq f(x \vee y) + f(x \wedge y)
\end{equation}
for all $x,y \in \ZZ^n$,
where $x \vee y$  and $x \wedge y$ denote the componentwise maximum and minimum
of the vectors $x$ and $y$, respectively.
As is well known, a function $f$ defined on a rectangular domain
is submodular if and only if it satisfies
the inequality (\ref{submfndef})
for every pair $(x, y)$
with $\| x - y \|_{\infty} = 1$.
Furthermore, since
$x \vee y = \left\lceil (x+y)/2  \right\rceil$ and
$x \wedge y = \left\lfloor (x+y)/2 \right\rfloor$
when $\| x - y \|_{\infty} = 1$,
the inequality (\ref{submfndef})
is equivalent to discrete midpoint convexity
(\ref{discmptconvfn}) for points at $\ell_\infty$-distance 1.
Thus, on a rectangular domain, the class of submodular functions
coincides with DMC($1$).

In discrete convex analysis \cite{Mdca98,Mdcasiam,Mbonn09,Mdcaeco16},
a variety of discrete convex functions are considered.
A function
$f: \ZZ^{n} \to \RR \cup \{ +\infty \}$
is called
{\em ${\rm L}^{\natural}$-convex}
if it satisfies
translation-submodularity:
\begin{equation} \label{lnatftrsubm0}
  f(x) + f(y) \geq f((x - \mu {\bf 1}) \vee y)
                 + f(x \wedge (y + \mu {\bf 1}))
\end{equation}
for all $x, y \in \ZZ^{n}$ and nonnegative integers $\mu$,
where $\bm{1}=(1,1,\ldots, 1)$.
It is known that $f$ is ${\rm L}^{\natural}$-convex
if and only if it satisfies
discrete midpoint convexity (\ref{discmptconvfn}) for all $x,y \in \ZZ^{n}$,
or, equivalently, for all points $x,y$ at $\ell_\infty$-distance 1 or 2
 (see Theorem \ref{THlnatcond}). 
 Thus, the class of ${\rm L}^{\natural}$-convex functions coincides with the two 
 equivalent classes 
DMC($\geq$1) = DMC(1) $\cap$ DMC(2).
${\rm L}^{\natural}$-convex functions
form a major class of discrete convex functions
and have applications in several fields including
image processing \cite{KS09lnatmin},
auction theory \cite{LLM06,MSY16auction},
inventory theory \cite{SCB14,Zip08}, and scheduling \cite{BQ11}.

Separable convex functions are a special case of
${\rm L}^{\natural}$-convex functions.
A function $f: \ZZ^{n} \to \RR \cup \{ +\infty \}$ in $x=(x_{1}, \ldots,x_{n})$
is called {\em separable convex} if it can be represented as
$f(x) = \varphi_{1}(x_{1}) + \cdots + \varphi_{n}(x_{n})$
with univariate discrete convex functions
$\varphi_{i}: \ZZ \to \RR \cup \{ +\infty \}$
satisfying discrete midpoint convexity at distance two:
$\varphi_{i}(t-1) + \varphi_{i}(t+1) \geq 2 \varphi_{i}(t)$ for all $t \in \ZZ$.

A function $f: \ZZ^{n} \to \RR \cup \{ +\infty \}$
is called {\em integrally convex}
\cite{FT90} if its local convex envelope
$\tilde{f}: \RR^{n} \to \RR \cup \{ +\infty \}$
defined by (\ref{fnconvclosureloc0})
is (globally) convex in the ordinary sense. Note that
$\tilde{f}$
can be viewed as the collection of convex envelopes of $f$ in each
unit hypercube
$\{ x \in \RR\sp{n} \mid a_{i} \leq x_{i} \leq a_{i} + 1 \ (i=1,\ldots, n) \}$
with $a \in \ZZ^{n}$.
Integrally convex functions
constitute a common framework for discrete convex functions,
including separable convex and
${\rm L}^{\natural}$-convex functions
as well as
${\rm M}^{\natural}$-convex,
${\rm L}^{\natural}_{2}$-convex and
${\rm M}^{\natural}_{2}$-convex functions \cite{Mdcasiam},
and BS-convex and UJ-convex functions \cite{Fuj14bisubdc}.
Submodular integrally convex functions are exactly
${\rm L}^{\natural}$-convex functions \cite{FM00}.
The concept of integral convexity is
used in formulating discrete fixed point theorems
\cite{Iim10,IMT05,Yan09fixpt}.
In game theory the integral concavity of payoff functions
guarantees the existence of a pure strategy equilibrium
in finite symmetric games \cite{IW14}.

In the case of a rectangular domain, integrally convex functions
were shown to coincide with the equivalent classes 
WDMC($2$) = WDMC($\geq$2) already in the 
seminal paper \cite{FT90} (see Theorem \ref{THfavtarProp33}).
In the following section we clarify that the same relation
holds in the case of integrally convex domains (to be defined below).
For more general domains we might have 
WDMC($2$) $\neq$ WDMC($\geq$2). However, as we show in Appendix A,
WDMC($\geq$2) always coincides with the class of integrally convex functions.
To complete the picture, note that weak discrete midpoint convexity 
(\ref{weakmidconv}) trivially holds for any function and for all points 
$x,y$ at $\ell_\infty$-distance 1. Thus WDMC($1$) contains all functions.

After describing the classes 
DMC($1$),    
DMC($\geq$1) = DMC($1$) $\cap$ DMC($2$), 
WDMC($1$),
and 
WDMC($\geq$1) = WDMC($1$) $\cap$ WDMC($2$), 
one might think of the classes 
DMC($k$) or DMC({$\geq$}{$k$}) or WDMC($k$)  or WDMC({$\geq$}{$k$}) for $k \geq 3$. 
However, it does not seem
interesting to consider 
such classes of functions that satisfy (weak) discrete midpoint convexity only for all
$x,y \in \ZZ^n$ with $\|x-y \, \|_\infty \geq 3$,
since this 
condition  is not even sufficient  
for the standard one-dimensional discrete convexity on $\ZZ$. 
For example, the function $g: \ZZ \to \RR$ with $g(0) = 2$ and  $g(z) = z\sp{2}$ for $z \neq 0$
satisfies discrete midpoint convexity for all
$x,y \in \ZZ$ with $\|x-y \, \|_\infty \geq 3$,
but is not discrete convex on $\ZZ$.

Thus, the only potentially interesting classes of discrete midpoint convex functions 
that have not yet been investigated are the classes DMC($2$) and DMC($\geq$2),
which we call \emph{locally discrete midpoint convex functions}%
\footnote{%%%%
Locally discrete midpoint convex functions have been previously called
directed integrally convex functions
in the preliminary paper \cite{MMTT16proxICissac}
which anticipates part of the results of this paper.
} %%%footnote%%%%%%%%%%%%% 
and  \emph{globally discrete midpoint convex functions}, respectively.
We analyze their properties in Sections \ref{SCdirintcnvfn} and 
\ref{SCdirintparaineq}--\ref{SCdirintcnvprox},
where we also prove that DMC($\geq$2) = DMC($2$) $\cap$ DMC($3$)
(Proposition~\ref{PRmdptdistant}),
thus showing that no intermediate class of discrete midpoint convex functions exists
between locally and globally midpoint convex functions.

%%%  FIGURE %%%%%%%%%%%%%%%%%%
\begin{figure}
\begin{center}
\begin{tabular}{|ll|}
\hline
 WDMC(1)
& {: all functions}
\\
 DMC(1)
& {: submodular}
\\
 WDMC(2) = WDMC($\geq$1)
& {: integrally convex}
\\
 DMC(1) $\cap$ WDMC(2)
& {: submodular integrally convex}
\\
= DMC(1) $\cap$ DMC(2) &
\\
= DMC($\geq$1)
& {: ${\rm L}^{\natural}$-convex}
\\
 $\subsetneqq$  DMC(2) $\cap$ DMC(3) = DMC($\geq$2)
&  {: globally discrete midpoint convex}
\\
 $\subsetneqq$  DMC(2)
&  {: locally discrete midpoint convex}
\\
 $\subsetneqq$  WDMC(2)
& {: integrally convex}
\\
\hline
\end{tabular}
\end{center}

\caption{Function classes defined by discrete midpoint convexity
($\dom f$: discrete rectangle)}
\label{FGmidptfnclass}
\end{figure}
%%%  FIGURE %%%%%%%%%%%%%%%%%%

%%%%%%%%%%%%%% SSSSS %%%%%%%%%%%%%%%%%%%%%
\subsection{Weak discrete midpoint convexity and integral convexity}
\label{SCintcnvfn}

Recall that a function
$f: \mathbb{Z}^{n} \to \mathbb{R} \cup \{ +\infty  \}$
is said to be {\em integrally convex}
\cite{FT90}
if its local convex envelope
$\tilde{f}: \RR^{n} \to \RR \cup \{ +\infty \}$
defined by (\ref{fnconvclosureloc0})
with respect to the integral neighborhood
$N(x) = \{ z \in \mathbb{Z}^{n} \mid | x_{i} - z_{i} | < 1 \ (i=1,\ldots,n)  \}$
$(x \in \RR^{n})$ is convex on $\RR^{n}$.

A set $S \subseteq \ZZ^{n}$ is said to be
integrally convex if
the convex hull $\overline{S}$ of $S$ coincides with the union of the
convex hulls of $S \cap N(x)$ over $x \in \RR^{n}$,
i.e., if, for any $x \in \RR^{n}$,
$x \in \overline{S} $ implies $x \in  \overline{S \cap N(x)}$.
A set $S \subseteq \ZZ^{n}$ is integrally convex if and only if
its indicator function $\delta_{S}$ is an integrally convex function,
where the indicator function $\delta_{S}$ is defined by
$\delta_{S}(x) = 0$ for $x \in S$
and $\delta_{S}(x) = + \infty$ for $x \not\in S$.
An integrally convex set is ``hole-free'' in the sense of
$S = \overline{S} \cap \ZZ\sp{n}$.
The effective domain
$\dom f = \{ x \in \ZZ^{n} \mid f(x) < +\infty \}$
and the set of minimizers of an integrally convex function
$f$ are both integrally convex \cite[Proposition 3.28]{Mdcasiam}.

Integral convexity can be characterized
by a local condition
under the assumption that the effective domain
is an integrally convex set.
The following theorem is proved in \cite[Proposition 3.3]{FT90}
when the effective domain
is an integer interval (discrete rectangle)
and in the general case in \cite[Theorem 2.3]{MMTT17proxIC}.

\begin{theorem}[{\cite{FT90,MMTT17proxIC}}] \label{THfavtarProp33}
Let $f: \mathbb{Z}^{n} \to \mathbb{R} \cup \{ +\infty  \}$
be a function with an integrally convex effective domain.
Then the following properties are equivalent:

{\rm (a)}
$f$ is integrally convex.

{\rm (b)}
For every $x, y \in \ZZ^{n}$ with $\| x - y \|_{\infty} =2$
we have \
\begin{equation}  \label{intcnvconddist2}
 f(x) + f(y) \geq 2 \tilde{f}\, \bigg(\frac{x + y}{2} \bigg).
\end{equation}
\vspace{-1.7\baselineskip}
\\
\finbox
\end{theorem}

Since integral convexity clearly implies weak discrete midpoint convexity
for any pair of points $x, y \in \ZZ^n$, and the condition
in part (b) of Theorem \ref{THfavtarProp33}
is the local version of weak discrete midpoint convexity that defines WDMC(2),
we obtain the following characterization of integral convexity in terms of
``local'' or ``global'' weak discrete midpoint convexity.

\begin{corollary} \label{Cor:WDMC-IntConv}
Let $f: \mathbb{Z}^{n} \to \mathbb{R} \cup \{ +\infty  \}$
be a function with an integrally convex effective domain.
Then the following properties are equivalent%
\footnote{%%%%%%%%%%%%%%%%
(a) and (c) are equivalent without the assumption of $\dom f$ being integrally convex;
see Appendix \ref{SCwdmcIC}.
}:  %%%footnote%%%%%%

{\rm (a)}
$f$ is integrally convex.

{\rm (b)}
$f \in$ {\rm WDMC(2)}.

{\rm (c)}
$f \in$ {\rm WDMC($\geq$1)}.
\finbox
\end{corollary}

One of the most useful properties of convex functions is the
local characterization of global minima. This property holds
for integrally convex functions in the following form.

%%%%%%%%%%
\begin{theorem}[\protect{\cite[Proposition 3.1]{FT90}};
  see also \protect{\cite[Theorem 3.21]{Mdcasiam}}]
  \label{THintcnvlocopt}
Let $f: \mathbb{Z}^{n} \to \mathbb{R} \cup \{ +\infty  \}$
be an integrally convex function and $x^{*} \in \dom f$.
Then $x^{*}$ is a minimizer of $f$
if and only if
$f(x^{*}) \leq f(x^{*} +  d)$ for all
$d \in  \{ -1, 0, +1 \}^{n}$.
\finbox
\end{theorem}

Theorem~\ref{THintcnvlocopt} above can be generalized to the following
``box-barrier property'' (see Fig.~\ref{FGboxbarrier}),
of which Theorem~\ref{THintcnvlocopt} is a special case
with $p=\hat x - \bm{1}$ and $q=\hat x + \bm{1}$.

%%%  FIGURE %%%%%%%%%%%%%%%%%%
\begin{figure}\begin{center}
\includegraphics[height=40mm]{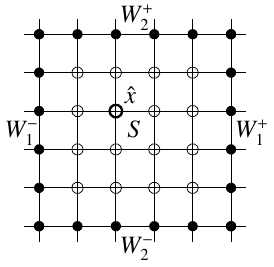}
\caption{Box-barrier property ($\circ \in S$, $\bullet \in W$)}
\label{FGboxbarrier}
\end{center}\end{figure}
%%%  FIGURE %%%%%%%%%%%%%%%%%%

\begin{theorem}[Box-barrier property \cite{MMTT16proxICissac,MMTT17proxIC}] \label{THintcnvbox}
Let $f: \mathbb{Z}^{n} \to \mathbb{R} \cup \{ +\infty  \}$
be an integrally convex function,
and let
$p \in (\mathbb{Z} \cup \{ -\infty  \})^{n}$
and
$q \in (\mathbb{Z} \cup \{ +\infty  \})^{n}$,
where $p \leq q$.
Define
\begin{align*}
S &= \{ x \in \mathbb{Z}^{n} \mid p_{i} < x_{i} < q_{i} \ (i=1,\ldots,n) \},
\\
W_{i}^{+} &= \{ x \in \mathbb{Z}^{n} \mid
 x_{i} = q_{i}, \  p_{j} \leq x_{j} \leq q_{j} \ (j \not= i) \}
\quad (i=1,\ldots,n),
\\
W_{i}^{-} &= \{ x \in \mathbb{Z}^{n} \mid
 x_{i} = p_{i}, \  p_{j} \leq x_{j} \leq q_{j} \ (j \not= i) \}
\quad (i=1,\ldots,n),
\end{align*}
and $W = \bigcup_{i=1}^{n} (W_{i}^{+} \cup W_{i}^{-})$.
Let $\hat x \in S \cap \dom f$.
If
$f(\hat x) \leq f(y)$ for all $y \in W$,
then
$f(\hat x) \leq f(z)$ for all $z \in \ZZ^{n} \setminus S$.
\finbox
\end{theorem}

%%%%%%%%%%%%%%%% SSSSS %%%%%%%%%%%%%%%%%%%%%
\subsection{Discrete midpoint convex sets}
\label{SCdirintcnvset}

We call a set $S \subseteq  \ZZ\sp{n}$ {\em discrete midpoint convex} if
\begin{equation} \label{dirintcnvsetdef}
 x, y \in S, \ \| x - y \|_{\infty} \geq 2
\ \Longrightarrow \
\left\lceil \frac{x+y}{2} \right\rceil ,
\left\lfloor \frac{x+y}{2} \right\rfloor  \in S.
\end{equation}
For comparison we recall \cite{Mdcasiam} that a set $S \subseteq  \ZZ\sp{n}$ is
called {\em ${\rm L}^{\natural}$-convex}
if
$\left\lceil (x+y)/2 \right\rceil, \left\lfloor (x+y)/2  \right\rfloor  \in S$
for all $x, y \in S$.

\begin{proposition}  \label{PRdirintcnvSetIC}
\label{PRdirintSetL}
\quad \
\\
{\rm (1)}
An ${\rm L}\sp{\natural}$-convex set
is discrete midpoint convex.
\\
{\rm (2)}
A discrete midpoint convex set is integrally convex.
\end{proposition}
\begin{proof}
(1) is obvious, whereas the proof of (2) is given in Section~\ref{SCdicproofPRdirintcnvSetIC}.
\end{proof}

The following proposition states some straightforward
properties of discrete midpoint convex sets.

\begin{proposition} \label{PRdirintcnvSeteasy}
\quad \
\\
{\rm (1)}
Every subset of $\{ 0, 1\}\sp{n}$ is discrete midpoint convex.
\\
{\rm (2)}
The intersection of two (or more) discrete midpoint convex sets is
discrete midpoint convex.
\finbox
\end{proposition}

\begin{remark} \rm  \label{RMdicinclusion}
The inclusion relations among the classes of
${\rm L}\sp{\natural}$-convex sets,
discrete midpoint convex sets, and
integrally convex sets
are all proper.  That is,
$\{ \mbox{${\rm L}\sp{\natural}$-convex sets} \}$
 $\subsetneqq$
$\{ \mbox{discrete midpoint convex sets} \}$
 $\subsetneqq$
$\{ \mbox{integrally convex sets} \}$.
For example, the set $\{ (1,0), (0,1)\}$
is a discrete midpoint convex set that is not ${\rm L}\sp{\natural}$-convex.
The set $\{ (x_{1},x_{2}) \in \ZZ\sp{2} \mid x_{1}+x_{2}=0 \}$
is an integrally convex set that is not discrete midpoint convex.
\finbox
\end{remark}

\begin{remark} \rm  \label{RMmpcL2}
A set $S \subseteq \ZZ\sp{n}$ is called
{\em ${L}^{\natural}_{2}$-convex}
if it can be represented as the Minkowski sum of two
${\rm L}^{\natural}$-convex sets.
There is no inclusion relation between
${\rm L}^{\natural}_{2}$-convex sets and discrete midpoint convex sets.
For example, the set $\{ (1,0), (0,1)\}$
is a discrete midpoint convex set that is not ${\rm L}^{\natural}_{2}$-convex.
The set $\{ (0,0,0,0), (0,1,1,0), (1,1,0,0), (1,2,1,0) \}$
is an ${\rm L}^{\natural}_{2}$-convex set
that is not discrete midpoint convex.
Indeed, this set is the Minkowski sum of two
${\rm L}^{\natural}$-convex sets
$\{ (0,0,0,0), (1,1,0,0) \}$ and $\{ (0,0,0,0), (0,1,1,0) \}$,
and (\ref{dirintcnvsetdef}) fails for $x=(0,0,0,0)$ and $y=(1,2,1,0)$.
\finbox
\end{remark}

Besides midpoint convexity,  a convex set $S$ in $\RR\sp{n}$
has the property that for two points $x$ and $y$ in $S$
and a direction vector $d = t(y-x)$ from $x$ to $y$
with $0 \leq t \leq 1$,
both $x+d$ and $y-d$ belong to $S$.
Theorem \ref{THdirintcnvSetdec} below
shows a discrete analogue of this property,
where the direction vector $d$ is discretized in a specific manner.

For any pair of distinct vectors $x, y \in \ZZ\sp{n}$,
we can consider a decomposition of $y-x$ into vectors of $\{ -1,0,+1 \}\sp{n}$ as
\begin{equation} \label{DICdecAkBksum}
y - x = \sum_{k=1}^{m} (\bm{1}_{A_{k}} - \bm{1}_{B_{k}}),
\end{equation}
where $m = \| y - x \|_{\infty}$,
$\bm{1}_{A}$ is the characteristic vector of $A \subseteq \{ 1,\ldots, n\}$,
\begin{equation} \label{DICdecAkBkdef}
 A_{k} = \{ i \mid y_{i}- x_{i} \geq m+1-k \},
\quad
 B_{k} = \{ i \mid y_{i}- x_{i} \leq -k \}
\qquad (k=1,\ldots,m) .
\end{equation}
Note that
$A_{1} \subseteq A_{2}  \subseteq \cdots \subseteq A_{m}$, \
$B_{1} \supseteq B_{2} \supseteq \cdots \supseteq B_{m}$, \
$A_{m} \cap B_{1} = \emptyset$, and
$A_{1} \cup B_{m} \not= \emptyset$.

\begin{example} \rm \label{EXvecdecAB}
Figure \ref{FGvecdecAB} shows the subsets $A_{k}$ and $B_{k}$
in (\ref{DICdecAkBkdef})
for $x=(0,0,0,0,0,0,0)$ and $y=(4,2,1,0,-1,-2,-2)$.
We have $m= 4$,
$(A_{1},A_{2},A_{3},A_{4})=( \{ 1 \}, \{ 1 \},  \{ 1,2 \}, \{ 1,2,3 \})$
and
$(B_{1},B_{2},B_{3},B_{4})=( \{ 5,6,7 \}, \{ 6,7 \},  \emptyset, \emptyset)$.
%% keep %%%
%% $(A_{1},B_{1})=( \{ 1 \}, \{ 5,6,7 \})$,  $(A_{2},B_{2})=( \{ 1 \}, \{ 6,7 \})$,
%% $(A_{3},B_{3})=( \{ 1,2 \}, \emptyset)$, $(A_{4},B_{4})=( \{ 1,2,3 \}, \emptyset)$.
%%%%%%%%%%%
Accordingly, the decomposition (\ref{DICdecAkBksum}) is given by
$y-x = (4,2,1,0,-1,-2,-2) =
(1,0,0,0,-1,-1,-1) +
(1,0,0,0,0,-1,-1) +
(1,1,0,0,0,0,0) +
(1,1,1,0,0,0,0)$.
\finbox
\end{example}

%%%  FIGURE %%%%%%%%%%%%%%%%%%
%%\input{MMTTfgvecdecAB}
\begin{figure}\begin{center}
\includegraphics[height=30mm]{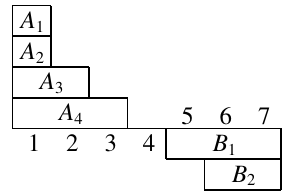}
\caption{Subsets $A_{k}$ and $B_{k}$ in (\ref{DICdecAkBkdef})
for $y-x = (4,2,1,0,-1,-2,-2)$ in Example \ref{EXvecdecAB}}
\label{FGvecdecAB}
\end{center}\end{figure}
%%%  FIGURE %%%%%%%%%%%%%%%%%%

\begin{theorem}  \label{THdirintcnvSetdec}
Let $S  \subseteq \mathbb{Z}\sp{n}$
be a discrete midpoint convex set, $x, y \in S$,
and  $d = \sum_{k \in J} (\bm{1}_{A_{k}} - \bm{1}_{B_{k}})$
for some
$J \subseteq \{ 1,2,\ldots, m \}$.
Then  $x+d \in S$ and $y-d \in S$.
\end{theorem}
\begin{proof}
The proof is given in Section~\ref{SCdicproofTHdirintcnvSetdec}.
\end{proof}

The property shown in Theorem \ref{THdirintcnvSetdec}
will be referred to as the ``parallelogram property''
of discrete midpoint convex sets,
as a parallelogram is formed by the four points $x$, $x+d$, $y$, and $y-d$
(cf.~Fig.~\ref{FGparaset}).
This property
will be used to establish the
``parallelogram inequality,''
a key property of discrete midpoint convex functions,
in Theorem \ref{THparallelineqgen}.

%%%  FIGURE %%%%%%%%%%%%%%%%%%
%%\input{MMTTfgparasetV2}
\begin{figure}\begin{center}
\includegraphics[height=25mm]{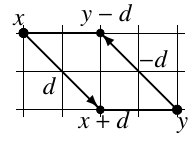}
\caption{Parallelogram property}
\label{FGparaset}
\end{center}\end{figure}
%%%  FIGURE %%%%%%%%%%%%%%%%%%

\begin{example} \rm \label{EXparaset}
For $x=(0,0)$ and $y=(4,-2)$,
the decomposition (\ref{DICdecAkBksum})
is given as
$y-x = (4,-2) = (1,-1) + (1,-1) + (1,0) + (1,0)$
with $m= 4$,
$(A_{1},B_{1})=(A_{2},B_{2})=( \{ 1 \}, \{ 2 \})$, and
$(A_{3},B_{3})=(A_{4},B_{4})=( \{ 1 \}, \emptyset)$.
Suppose that $x$ and $y$ are contained in a discrete midpoint convex set $S$.
Different choices of
$J \subseteq \{ 1,2,3,4 \}$ in Theorem \ref{THdirintcnvSetdec}
generate different points $x+d, y-d$ contained in $S$,
as illustrated in Fig.~\ref{FGparasetEX}.
\finbox
\end{example}

%%%  FIGURE %%%%%%%%%%%%%%%%%%
%%\input{MMTTfgparasetEX}
\begin{figure}\begin{center}
\includegraphics[width=\textwidth]{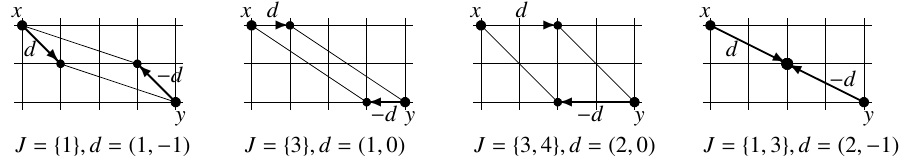}
\caption{Parallelogram property for $x= (0,0)$, $y=(4,-2)$
in Example \ref{EXparaset}}
\label{FGparasetEX}
\end{center}\end{figure}
%%%  FIGURE %%%%%%%%%%%%%%%%%%

%%%%%%%%%%%%%% SSSSS %%%%%%%%%%%%%%%%%%%%%
\subsection{Discrete midpoint convex functions}
\label{SCdirintcnvfn}

In this section we consider discrete midpoint convexity
\begin{equation} \label{discmptconvfn2}
 f(x) + f(y) \geq
    f \left(\left\lceil \frac{x+y}{2} \right\rceil\right)
  + f \left(\left\lfloor \frac{x+y}{2} \right\rfloor\right)
\end{equation}
for functions $f: \ZZ\sp{n} \to \RR \cup \{ +\infty \}$.

In Section \ref{SCnewandoldDMC}
we have defined
${\rm L}^{\natural}$-convex functions
in terms of translation-submodularity (\ref{lnatftrsubm0}).
However, ${\rm L}^{\natural}$-convexity can be viewed as a type 
of discrete midpoint convexity by the following known result.

\begin{theorem}
[\protect{\cite{FT90, FM00,Mdcasiam}}]
 \label{THlnatcond}
For $f: \ZZ^{n} \to \RR \cup \{ +\infty \}$
the following conditions, {\rm (a)} to {\rm (d)}, are equivalent:%
\footnote{%%%%%%%%%%%%%%%%%%%%
$\ZZ$-valued functions are treated in \cite[Theorem~3]{FM00},
but the proof is valid for $\RR$-valued functions.
} %%%%%%%%%%%%%%%%%%%%%%%

{\rm (a)}
$f$ satisfies translation-submodularity
{\rm (\ref{lnatftrsubm0})}.

{\rm (b)}
$f$ satisfies discrete midpoint convexity
{\rm (\ref{discmptconvfn2})} for all $ x, y \in \ZZ^{n}$,
i.e., $f \in \mbox{\rm DMC($\geq$1)}$.

{\rm (c)}
$f$ satisfies
discrete midpoint convexity
{\rm (\ref{discmptconvfn2})} for all $ x, y \in \ZZ^{n}$
with $\| x-y \|_{\infty} \leq 2$,
i.e., $f \in \mbox{\rm DMC}(1) \cap \mbox{\rm DMC}(2)$,
 and
the effective domain has the property:
$x, y \in \dom f   \Rightarrow
   \left\lceil  (x+y)/2 \right\rceil ,
   \left\lfloor (x+y)/2 \right\rfloor \in \dom f $.

{\rm (d)}
$f$ is integrally convex and submodular%
\footnote{%%%%%%%%%%%%%
If $\dom f$ is a discrete rectangular domain (box),  
condition (d) can be written ``$f \in \mbox{\rm DMC}(1) \cap \mbox{\rm WDMC}(2)$''
as in Fig.~\ref{FGmidptfnclass}.
}. %%%% footnote %%%%%%
\finbox
\end{theorem}

We call a function $f: \ZZ\sp{n} \to \RR \cup \{ +\infty \}$
{\em  locally discrete midpoint convex}
if $\dom f$ is a discrete midpoint convex set
and the discrete midpoint convexity (\ref{discmptconvfn2})
is satisfied by every pair $(x, y) \in \ZZ\sp{n} \times \ZZ\sp{n}$
with $\| x - y \|_{\infty} = 2$ (exactly equal to $2$).
We call a function $f: \ZZ\sp{n} \to \RR \cup \{ +\infty \}$
{\em  globally discrete midpoint convex} if
the discrete midpoint convexity (\ref{discmptconvfn2})
is satisfied by every pair $(x, y) \in \ZZ\sp{n} \times \ZZ\sp{n}$
with $\| x - y \|_{\infty} \geq 2$.
The effective domain of a
globally discrete midpoint convex function
is necessarily a discrete midpoint convex set.

\begin{proposition} \label{PRstrdom}
For a globally discrete midpoint convex function $f$,
the effective domain
$\dom f$ is a discrete midpoint convex set.
\end{proposition}
\begin{proof}
It follows from (\ref{discmptconvfn2})
for $\| x - y \|_{\infty} \geq 2$
that
$S = \dom f$ satisfies (\ref{dirintcnvsetdef}).
\end{proof}

We have the following inclusion relations among the function classes:
\begin{align}
& \{ \mbox{ ${\rm L}\sp{\natural}$-convex } \}
= \{ \mbox{ submodular integrally convex } \}
\notag \\
& \subsetneqq \
\{ \mbox{ globally discrete midpoint convex  } \}
\notag \\
& \subsetneqq \
\{ \mbox{ locally discrete midpoint convex } \}
\notag \\
&
\subsetneqq \
\{ \mbox{ integrally convex } \},
  \label{fnclasses3}
\end{align}
as proved below.
Using the notation of the Introduction, we can express (\ref{fnclasses3}) as
$\mbox{\rm DMC($\geq$1)}
\subsetneqq \mbox{\rm DMC($\geq$2)}
\subsetneqq \textrm{DMC}(2)
\subsetneqq \mbox{\rm WDMC($\geq$1)}$.

\begin{theorem}  \label{THdirintclass}
\quad \
\\
{\rm (1)}
An ${\rm L}\sp{\natural}$-convex function
is globally discrete midpoint convex.
\\
{\rm (2)}
A globally discrete midpoint convex function
is locally discrete midpoint convex.
\\
{\rm (3)}
A locally discrete midpoint convex function is integrally convex.
\end{theorem}

\begin{proof}
(1) This is immediate from the characterizations of
${\rm L}\sp{\natural}$-convex functions by discrete midpoint convexity
in Theorem~\ref{THlnatcond}.

(2)
Let $f$ be a globally discrete midpoint convex function.
By Proposition~\ref{PRstrdom}, $\dom f$ is a discrete midpoint convex set.
The discrete midpoint convexity (\ref{discmptconvfn2})
when
$\| x - y \|_{\infty} =2$
is obviously true, since it is true when
$\| x - y \|_{\infty}  \geq 2$.

(3)
Let $f$ be a locally discrete midpoint convex function.
Then $\dom f$ is a discrete midpoint convex set, which
is an integrally convex set by Proposition~\ref{PRdirintcnvSetIC}.
We use the characterization of integrally convex functions
stated in Theorem~\ref{THfavtarProp33}.
Recall that,
for $u \in \RR\sp{n}$, $\tilde{f}(u)$ denotes the local convex extension of
$f$ with respect to the integer neighborhood $N(u)$.
Take integer points $x, y \in \dom f$ with $\| x - y \|_{\infty} =2$.
For $u = (x+y)/2$ ,
both
$\left\lceil u \right\rceil$
and
$\left\lfloor u \right\rfloor$
belong to $N(u)$, and therefore
$2 \tilde{f}(u) \leq f( \left\lceil u \right\rceil ) + f( \left\lfloor u \right\rfloor )
\leq f(x) + f(y)$,
where the second inequality is due to (\ref{discmptconvfn2}).
Thus, $f$ satisfies  condition (b) in Theorem~\ref{THfavtarProp33},
and therefore it is integrally convex.
\end{proof}

The inclusions in (\ref{fnclasses3}) are all proper.
The examples in Remark~\ref{RMdicinclusion}
in Section~\ref{SCdirintcnvset}
show this for the first and third
``$\subsetneqq$'' since
a set $S \subseteq  \ZZ\sp{n}$
is discrete midpoint convex if and only if
its indicator function $\delta_{S}$ is
a (globally or locally) discrete midpoint convex function.
The difference between global and local versions
is demonstrated in the following examples.

\begin{example} \rm \label{EXmdpt1}
$f(x_{1}, x_{2})= | x_{1} + x_{2} |$
is a locally discrete midpoint convex function on
$\dom f = \ZZ\sp{2}$,
satisfying the discrete midpoint convexity (\ref{discmptconvfn2})
when $\| x - y \|_{\infty} = 2$.
For $x=(0,0)$ and $y=(3,-3)$
with $\| x - y \|_{\infty} = 3$
we have
$\left\lceil (x+y)/2 \right\rceil = (2,-1)$
and
$\left\lfloor (x+y)/2 \right\rfloor = (1,-2)$, and therefore,
$ f(x) + f(y) = 0 + 0$,
$   f \left(\left\lceil \frac{x+y}{2} \right\rceil\right)
  + f \left(\left\lfloor \frac{x+y}{2} \right\rfloor\right)
= 1 + 1$.
This shows a failure of (\ref{discmptconvfn2}).
Hence $f$ is not globally discrete midpoint convex.
\finbox
\end{example}

\begin{example} \rm \label{EXdicdim2}
Let $f(x) = x\sp{\top} Q x$ with
$Q = {\small
\left[ \begin{array}{cc}
1 & c  \\
c & 1 \\
\end{array}\right]}$
for $x  \in \ZZ\sp{2}$.
It can be verified from the definitions
(or by Propositions \ref{PRdiagdomLnat}, \ref{LMdiagDomQuad2dimG}, and \ref{LMdiagDomQuad2dimL}
in Section~\ref{SCquadfnDIC}) 
that
\begin{itemize}
\item
$f$ is ${\rm L}\sp{\natural}$-convex
if and only if $-1 \leq c \leq 0$,

\item
$f$ is globally discrete midpoint convex
if and only if $-1 \leq c \leq 4/5$,

\item
$f$ is locally discrete midpoint convex
if and only if $-1 \leq c \leq 1$.
\finbox
\end{itemize}
\end{example}

Further relations among the function classes in some special cases are shown next.

\begin{proposition} \label{PRdirintclass2}
\quad \
\\
{\rm (1)}
Every function on the unit hypercube $\{ 0, 1\}\sp{n}$ is
globally (and hence locally) discrete midpoint convex.
\\
{\rm (2)}
For $n=2$, every integrally convex function on a discrete midpoint convex set
is locally discrete midpoint convex%
\footnote{%%%%%%%%%%%%%%%%%
Example \ref{EXmdpt1} shows that
we cannot replace ``locally'' with ``globally'' in this statement.
}. %% footnote %%%%%%%%%%%%
\\
{\rm (3)}
Let $f$ be a submodular function on $\ZZ\sp{n}$.
Then $f$ is integrally convex
if and only if it is (globally or locally) discrete midpoint convex.
\end{proposition}
\begin{proof}
(1)
There exist no $x, y \in \{ 0, 1\}\sp{n}$ with $\| x - y \|_{\infty} =2$.
Therefore, the condition for global discrete midpoint convexity is
satisfied vacuously.

(2)
Suppose that $f$ is integrally convex, and take
any $x, y \in \ZZ\sp{2}$ with  $\| x - y \|_{\infty} =2$, and
let $u  = (x+y)/2$.
Since $n=2$,
$N \left( u \right) = \{ \left\lceil u \right\rceil, \left\lfloor u \right\rfloor \}$
if $u \not\in \ZZ\sp{2}$
and $N \left( u \right) = \{ u \}$ if $u \in \ZZ\sp{2}$.
Hence, the integral convexity (\ref{intcnvconddist2}) in Theorem \ref{THfavtarProp33}
implies the discrete midpoint convexity (\ref{discmptconvfn2}).

(3)
By Theorem \ref{THdirintclass} or (\ref{fnclasses3}),
a (globally or locally) discrete midpoint convex function
is integrally convex, irrespective of submodularity.
If $f$ is a submodular function that is integrally convex,
then it is ${\rm L}\sp{\natural}$-convex by Theorem~\ref{THlnatcond},
and hence
globally discrete midpoint convex by Theorem~\ref{THdirintclass} (1).
This completes the proof of (3), since global discrete midpoint convexity
implies local discrete midpoint convexity by
Theorem \ref{THdirintclass} (2).
\end{proof}

In the case of $n=2$ the inclusion relations in (\ref{fnclasses3}) are modified to 
\begin{align}
& \{ \mbox{ ${\rm L}\sp{\natural}$-convex } \}
\subsetneqq \
\{ \mbox{ globally discrete midpoint convex } \}
\notag \\
& \subsetneqq \
\{ \mbox{ locally discrete midpoint convex } \}
=
\{ \mbox{ integrally convex } \}
  \label{fnclasses2dim}
\end{align}
by Example \ref{EXdicdim2} and Proposition \ref{PRdirintclass2}(2).

For discrete midpoint convexity of
the minimizers we have the following statements.

\begin{proposition} \label{PRdirintcnvSetFn}
For a globally discrete midpoint convex function $f$,
the set of the minimizers
$\argmin f$ is a discrete midpoint convex set.
\end{proposition}
\begin{proof}
It follows from (\ref{discmptconvfn2}) that
$S = \argmin f$ satisfies (\ref{dirintcnvsetdef}).
\end{proof}

\begin{remark} \rm \label{RMwdiminimizer}
The set of the minimizers
of a locally discrete midpoint convex function $f$
is not necessarily a discrete midpoint convex set.
For example,
$f(x_{1}, x_{2})= | x_{1} + x_{2} |$
is locally discrete midpoint convex on $\ZZ\sp{2}$,
but $\argmin f = \{ (t,-t) \mid  t \in \ZZ \}$
is not a discrete midpoint convex set.
It should be clear that no local version
of discrete midpoint convexity is defined for sets.
\finbox
\end{remark}

Simple operations valid for globally or locally discrete midpoint convex functions
are listed below.
The scaling operation
will be treated separately in Section~\ref{SCscalingDIC}.

\begin{proposition}  \label{PRdirintcnvinvarS}
Let $f: \mathbb{Z}\sp{n} \to \mathbb{R} \cup \{ +\infty  \}$
be globally discrete midpoint convex functions.

\noindent
{\rm (1)}
For any $z \in \ZZ\sp{n}$,
 $f(z + x)$ is globally discrete midpoint convex in $x$.

\noindent
{\rm (2)}
For any permutation  $\sigma$ of $(1,2,\ldots,n)$,
 $f(x_{\sigma(1)}, x_{\sigma(2)}, \ldots, x_{\sigma(n)})$
is globally discrete midpoint convex in $x$.

\noindent
{\rm (3)}
 $f(- x_{1},- x_{2}, \ldots, - x_{n})$
is globally discrete midpoint convex in $x$.

\noindent
{\rm (4)}
For any $a_{1}, a_{2} \geq 0$ and
globally discrete midpoint convex functions
$f_{1}, f_{2}: \mathbb{Z}\sp{n} \to \mathbb{R} \cup \{ +\infty  \}$,
function $g= a_{1} f_{1}+ a_{2} f_{2}$ is globally discrete midpoint convex.
That is, the class of globally discrete midpoint convex functions
forms a convex cone.
\end{proposition}

\begin{proof}
(1)--(3) Obvious.
(4) The inequality (\ref{discmptconvfn2}) for $g$ follows immediately from
adding (\ref{discmptconvfn2}) for $a_{1} f_{1}$ and $a_{2} f_{2}$.
\end{proof}

\begin{proposition}  \label{PRdirintcnvinvarW}
The statements {\rm (1)}--{\rm (4)} in Proposition {\rm \ref{PRdirintcnvinvarS}}
 are true for locally discrete midpoint convex functions.
\end{proposition}

\begin{proof}
(1)--(3) Obvious.
(4) The effective domain $\dom g$ is discrete midpoint convex,
since $\dom g = \dom f_{1} \cap \dom f_{2}$
and the intersection of two discrete midpoint convex sets
is  discrete midpoint convex (Proposition~\ref{PRdirintcnvSeteasy} (2)).
The rest of the proof is the same as that of
Proposition~\ref{PRdirintcnvinvarS} (4).
\end{proof}

\begin{remark} \rm \label{RMdmcsigninver}
For a (globally or locally) discrete midpoint convex function $f$,
the function with individual sign inversion of variables,
$f(s_{1} x_{1}, s_{2} x_{2}, \ldots, s_{n} x_{n})$
with $s_{i} \in \{ +1, -1 \}$ $(i=1,2,\ldots,n)$,
is not necessarily
(globally or locally) discrete midpoint convex in $x$,
though it remains integrally convex in $x$.
For example,
$f(x_{1},x_{2},x_{3}) = \max( x_{1},x_{2},x_{3} )$
is globally discrete midpoint convex,
but
$g(x_{1}, x_{2},x_{3}) = f(x_{1}, x_{2},-x_{3}) = \max( x_{1},x_{2}, -x_{3} )$
is not even locally discrete midpoint convex,
since it fails to satisfy the inequality (\ref{discmptconvfn2}) for $x=(0,0,0)$ and $y=(-2,-1,1)$.
Indeed,  we have
$\left\lceil (x+y)/2 \right\rceil = (-1,0,1)$ and,
$\left\lfloor (x+y)/2 \right\rfloor = (-1,-1,0)$, and therefore,
$ g(x) + g(y) = 0 + (-1)$,
$    g \left(\left\lceil \frac{x+y}{2} \right\rceil\right)
  + g \left(\left\lfloor \frac{x+y}{2} \right\rfloor\right)
= 0 + 0 $.
\finbox
\end{remark}

\begin{remark} \rm %%\label{RM}
The concept of {\rm L}-extendable functions \cite{Hir15Lextprox}	
is defined for functions on the product of trees
in terms of some variant of discrete midpoint convexity.
When specialized to functions on $\ZZ\sp{n}$, a function
$f: \ZZ^{n} \to \RR \cup \{ +\infty \}$ is {\rm L}-extendable if and only if
$f(x) = g(2x)$ for some UJ-convex \cite{Fuj14bisubdc} function
$g: \ZZ^{n} \to \RR \cup \{ +\infty \}$.
The class of {\rm L}-extendable functions
does not contain, nor is contained by, the class of
(globally or locally) discrete midpoint convex functions,
as demonstrated by Examples \ref{EXdicButNOTlext} and  \ref{EXlextButNOTdic} below.
In the case of $n=2$, however, 
the class of {\rm L}-extendable functions
coincides with that of locally discrete midpoint convex functions (see also (\ref{fnclasses2dim})).
\finbox
\end{remark}

\begin{example} \rm \label{EXdicButNOTlext}
The function
$f: \ZZ^{3} \to \RR \cup \{ +\infty \}$
defined by
$f(1,0,0) = f(0,1,0) = f(0,0,1)=0$
with
$\dom f = \{ (1,0,0), (0,1,0), (0,0,1) \}$
is not {\rm L}-extendable by \cite[Theorem 1]{HI16ksubm},
whereas it is globally discrete midpoint convex by
Proposition~\ref{PRdirintclass2}(1).
\finbox
\end{example}

\begin{example} \rm \label{EXlextButNOTdic}
The function $f: \ZZ^{3} \to \RR$ defined by
$f(x_{1},x_{2},x_{3}) = \max( x_{1},x_{2},-x_{3} )$
is {\rm L}-extendable, whereas it is not locally discrete midpoint convex.
The {\rm L}-extendability follows from the facts that
(i) $g(x_{1},x_{2},x_{3}) = \max( x_{1},x_{2},x_{3} )$
is ${\rm L}\sp{\natural}$-convex,
(ii) every ${\rm L}\sp{\natural}$-convex function is {\rm L}-extendable, and
(iii) if a function is {\rm L}-extendable, so is
the function with individual sign inversion of variables.
However, the discrete midpoint convexity (\ref{discmptconvfn2}) fails
for $f$ with $x=(0,0,0)$ and $y=(-2,-1,1)$;
see  Remark \ref{RMdmcsigninver}.
\finbox
\end{example}

\begin{remark} \rm \label{RMtwosepconv}
For a univariate discrete convex function%
\footnote{%%%%%%%%%%%%%%%%%%
A univariate discrete convex function means a function  
$\varphi: \mathbb{Z} \to \mathbb{R} \cup \{ +\infty \}$
such that $\varphi(t-1) + \varphi(t+1) \geq 2 \varphi(t)$
for all $t \in \mathbb{Z}$.
} %%footnote %%%
$\varphi: \mathbb{Z} \to \mathbb{R} \cup \{ +\infty \}$,
we call {\em diff-convex} and {\em sum-convex} 
the functions of two variables
of the form $\varphi(s-t)$ and $\varphi(s+t)$, respectively,
where $(s,t) \in \mathbb{Z}\sp{2} $. 
It is straightforward to verify that diff-convex and sum-convex functions 
are locally discrete midpoint convex.
A function 
$f: \ZZ^{n} \to \RR \cup \{ +\infty \}$
is called
{\em 2-separable convex}
\cite{Hir15Lextprox}
if it can be expressed as the sum of univariate convex,
diff-convex and sum-convex functions, i.e., if
\begin{equation}\label{twosepSDconv}
f(x_1, \ldots, x_n) = 
\sum_{i=1}\sp{n} \varphi_{i}(x_{i}) 
+ \sum_{i \neq j}  \varphi_{ij}(x_i - x_j) + \sum_{i \neq j} \psi_{ij}(x_i+x_j)  ,
\end{equation}
where
$\varphi_{i}, \varphi_{ij}, \psi_{ij}: \mathbb{Z} \to \mathbb{R} \cup \{ +\infty \}$
$(i, j =1,\ldots,n; \  i \not = j)$
are univariate convex functions.
A function $f$ is called
{\em 2-separable diff-convex}
if the sum-convex terms are not involved in (\ref{twosepSDconv}), i.e., if
\begin{equation}\label{twoDiffConv}
f(x_1, \ldots, x_n) = 
\sum_{i=1}\sp{n} \varphi_{i}(x_{i}) 
+ \sum_{i \neq j}  \varphi_{ij}(x_i - x_j) .
\end{equation}
It is known that a 2-separable diff-convex function is ${\rm L}\sp{\natural}$-convex,
\cite{Mdcasiam,Mbonn09} and a
2-separable convex function is {\rm L}-extendable \cite{Hir15Lextprox}.
%%%% discussion with H  2017-05-04
However, a 2-separable convex function is not necessarily locally discrete midpoint convex.
For example, consider
$f(x_1, x_2, x_3) = (x_1 + x_2)^2$ for $n=3$.
This is a 2-separable convex function that is not locally discrete midpoint convex.
Indeed, for $x = (1,0,0)$ and $y = (0,1,2)$, we have
$z=\left\lceil \frac{x+y}{2} \right\rceil = (1,1,1)$,
$w = \left\lfloor \frac{x+y}{2} \right\rfloor = (0,0,1)$,
and $f(x) + f(y) = 1 + 1  < f(z) + f(w) = 4 + 0$.
\finbox
\end{remark}

%%%%%%%%%%%%%% SSSSS %%%%%%%%%%%%%%%%%%%%%
\section{Parallelogram Inequality}
\label{SCdirintparaineq}

An interesting inequality,
to be named 
``parallelogram inequality,''
can be established for (globally or locally)
discrete midpoint convexity functions.

\begin{theorem}[Parallelogram inequality]  \label{THparallelineqgen}
Let $f : \ZZ\sp{n} \to \RR \cup \{+\infty\}$ be 
a (globally or locally) discrete midpoint convex function, 
$x \in \dom f$, 
and 
$\{ A_{1}, A_{2}, \ldots, A_{m} ; B_{1},  B_{2}, \ldots, B_{m} \}$
be a multiset of subsets
of $\{ 1,2,\ldots, n \}$
such that
$A_{1} \subseteq A_{2}  \subseteq \cdots \subseteq A_{m}$, \
$B_{1} \supseteq B_{2} \supseteq \cdots \supseteq B_{m}$, \ 
$A_{m} \cap B_{1} = \emptyset$, \  $A_{1} \cup B_{m} \not= \emptyset$.
For any bipartition $(I,J)$ of $\{ 1,2,\ldots, m \}$
$(I \cap J = \emptyset$ and
$I \cup J = \{ 1,2,\ldots, m \})$
we have
\begin{align}
 f(x) + f(x + d_{1} + d_{2} )  \geq  f(x + d_{1}) +  f(x + d_{2}) ,
\label{paraineqd1d2gen}
\end{align}
where $\displaystyle d_{1} = \sum_{i \in I} (\bm{1}_{A_{i}} - \bm{1}_{B_{i}})$
and  $\displaystyle d_{2} = \sum_{j \in J} (\bm{1}_{A_{j}} - \bm{1}_{B_{j}})$.
\end{theorem}
\begin{proof}
For
$y = x + d_{1} + d_{2} 
= x+ \sum_{k=1}\sp{m} (\bm{1}_{A_{k}} - \bm{1}_{B_{k}})$
we may assume $y \in \dom f$, since otherwise 
(\ref{paraineqd1d2gen}) is trivially true
with $f(x + d_{1} + d_{2} ) = +\infty$.
Let
$d_{1}\sp{1},\ldots,d_{1}\sp{|I|}$
denote the vectors 
$\bm{1}_{A_{i}} - \bm{1}_{B_{i}}$ 
($i \in I$),
and similarly,
$d_{2}\sp{1},\ldots,d_{2}\sp{|J|}$
 the vectors 
$ \bm{1}_{A_{j}} - \bm{1}_{B_{j}}$ 
$(j \in J)$.
For $k=0,1,\ldots,|I|$ and $l=0,1,\ldots,|J|$ \ define
\[
 x(k,l) =  x + \sum_{i=1}\sp{k}  d_{1}\sp{i} 
             + \sum_{j=1}\sp{l} d_{2}\sp{j} .
\]
By Theorem~\ref{THdirintcnvSetdec},
we have $x(k,l) \in \dom f$ for
all $(k,l)$ with $0 \leq k \leq |I|$ and $0 \leq l \leq |J|$.
Note that (\ref{paraineqd1d2gen}) is rewritten as
\begin{equation}  \label{paraineqd1d2genxKL}
  f(x(0,0)) + f(x(|I|,|J|))   \geq  f(x(|I|,0)) +  f(x(0,|J|)) .
\end{equation}

Let $1 \leq k \leq |I|$ and $1 \leq l \leq |J|$.
We make use of the following properties of the vectors
$d_{1}\sp{k}$ and $d_{2}\sp{l}$:
\[
\| d_{1}\sp{k} + d_{2}\sp{l} \|_{\infty} = 2,
\qquad
\left\lceil (d_{1}\sp{k} + d_{2}\sp{l})/{2} \right\rceil
= \max(d_{1}\sp{k},d_{2}\sp{l}),
\qquad
\left\lfloor (d_{1}\sp{k} + d_{2}\sp{l})/{2} \right\rfloor
= \min(d_{1}\sp{k},d_{2}\sp{l}),
\]
where $\max(\cdot, \cdot)$ and $\min(\cdot, \cdot)$
denote the vectors of componentwise maximum and minimum, respectively.
This follows from the fact that
$d_{1}\sp{k} = \bm{1}_{A_{i}} - \bm{1}_{B_{i}}$ for some $i =i_{k} \in I$
and
$d_{2}\sp{l} = \bm{1}_{A_{j}} - \bm{1}_{B_{j}}$ for some $j =j_{l} \in J$,
where
$(A_{i} \cap A_{j}) \cup (B_{i} \cap B_{j})$
is nonempty and either
[$A_{i} \subseteq A_{j}$,  $B_{i} \supseteq B_{j}$, $A_{j} \cap B_{i} = \emptyset$]
or 
[$A_{j} \subseteq A_{i}$, $B_{j} \supseteq B_{i}$,  $A_{i} \cap B_{j} = \emptyset$]
according as $i < j$ or $i > j$.

Since 
$\| x(k,l) - x(k-1,l-1) \|_{\infty} = \| d_{1}\sp{k} + d_{2}\sp{l} \|_{\infty} = 2$,
we can use the discrete midpoint convexity (\ref{discmptconvfn2}) to obtain
\begin{align}
 & f(x(k-1,l-1)) +  f(x(k,l)) 
\notag \\  &\geq 
  f(\left\lceil (x(k-1,l-1)+x(k,l))/2 \right\rceil) 
+ f(\left\lfloor (x(k-1,l-1)+x(k,l))/2 \right\rfloor) .
\label{paraeqproof2}
\end{align}
By $x(k,l) =x(k-1,l-1) + d_{1}\sp{k} + d_{2}\sp{l}$ 
we have
\begin{align*}
&\lceil (x(k-1,l-1)+x(k,l))/2 \rceil
\\
& = x(k-1,l-1) + 
\left\lceil (d_{1}\sp{k} + d_{2}\sp{l})/{2} \right\rceil
\\
&= 
x(k-1,l-1) +  \max(d_{1}\sp{k},d_{2}\sp{l})
\\
& = 
 \left\{\begin{array}{rl}
  x(k-1,l-1) + d_{1}\sp{k} = x(k,l-1)&   (i_{k} > j_{l}) , \\
  x(k-1,l-1) + d_{2}\sp{l} = x(k-1,l) & (i_{k} < j_{l}), \\
    \end{array}\right.
\\
& \lfloor (x(k-1,l-1)+x(k,l))/2 \rfloor
\\
& = x(k-1,l-1) + 
\left\lfloor (d_{1}\sp{k} + d_{2}\sp{l})/{2} \right\rfloor
\\
&= 
x(k-1,l-1) +  \min(d_{1}\sp{k},d_{2}\sp{l})
\\
& = 
 \left\{\begin{array}{rl}
  x(k-1,l-1) + d_{2}\sp{l} = x(k-1,l)&   (i_{k} > j_{l}) , \\
  x(k-1,l-1) + d_{1}\sp{k} = x(k,l-1) & (i_{k} < j_{l}). \\
    \end{array}\right.
\end{align*}
This shows that the right-hand side of (\ref{paraeqproof2}) is equal to
$f(x(k,l-1)) +  f(x(k-1,l))$.  
Therefore, 
\[
 f(x(k-1,l-1)) +  f(x(k,l)) \geq  f(x(k,l-1)) +  f(x(k-1,l)).
\]
\par
By adding these inequalities for $(k,l)$ with
$1 \leq k \leq |I|$ and $1 \leq l \leq |J|$,
we obtain (\ref{paraineqd1d2genxKL}).
It is emphasized that all the terms that are cancelled
in this addition of inequalities are finite-valued,
since $x(k,l) \in \dom f$
for all $(k,l)$ with $0 \leq k \leq |I|$ and $0 \leq l \leq |J|$, 
by Theorem~\ref{THdirintcnvSetdec}.
\end{proof}

Theorem \ref{THparallelineqgen}
can be recast into the following form,
which is more convenient in some applications.
See also Figs.~\ref{FGparaset} and \ref{FGparasetEX}.

\begin{theorem}  \label{THparallelineqgen-var}
Let $f : \ZZ\sp{n} \to \RR \cup \{+\infty\}$ be 
a (globally or locally) discrete midpoint convex function, 
$x, y \in \dom f$, and 
 $d = \sum_{k \in J} (\bm{1}_{A_{k}} - \bm{1}_{B_{k}})$
for some $J \subseteq \{ 1,2,\ldots, m \}$ in {\rm (\ref{DICdecAkBksum})}.
Then
\begin{align}
 f(x) + f(y)  \geq  f(x + d) +  f(y - d) .
\label{paraineqxdyd}
\end{align}
\vspace{-1.5\baselineskip}\\
\finbox
\end{theorem}

As an application of the parallelogram inequality
we show some properties of 
locally discrete midpoint convex functions.
By definition, they
satisfy inequality (\ref{discmptconvfn2})
for $x$ and $y$ with $\| x - y \|_{\infty} = 2$.
Here we are interested in this inequality
for distant pairs $(x,y)$ with $\| x - y \|_{\infty} \geq 3$.
Example~\ref{EXmdpt1} shows that 
(\ref{discmptconvfn2}) is not always satisfied
if $\| x - y \|_{\infty} \geq 3$.

\begin{proposition}   \label{PRmdptdistant}
Let $f: \mathbb{Z}\sp{n} \to \mathbb{R} \cup \{ +\infty  \}$
 be a locally discrete midpoint convex function.

\noindent
{\rm (1)}
Discrete midpoint convexity {\rm (\ref{discmptconvfn2})} holds 
for all $x, y \in \ZZ\sp{n}$ with  $x \leq y$.

\noindent
{\rm (2)}
Discrete midpoint convexity {\rm (\ref{discmptconvfn2})} holds
for all $x, y \in \ZZ\sp{n}$ with  $\| x - y \|_{\infty}$ even.

\noindent
{\rm (3)}
If, in addition,  {\rm (\ref{discmptconvfn2})} holds
for all $x, y \in \ZZ\sp{n}$ with  $\| x - y \|_{\infty} = 3$, then 
{\rm (\ref{discmptconvfn2})} holds
for all $x, y \in \ZZ\sp{n}$ with  $\| x - y \|_{\infty} \geq 2$,
that is,
$f$ is globally discrete midpoint convex.
\end{proposition}
\begin{proof}
Consider the representation
$y - x = \sum_{k=1}\sp{m} (\bm{1}_{A_{k}} - \bm{1}_{B_{k}})$
as in (\ref{DICdecAkBksum}),
where $m = \| x - y \|_{\infty}$, 
$A_{1} \subseteq A_{2}  \subseteq \cdots \subseteq A_{m}$, \
$B_{1} \supseteq B_{2} \supseteq \cdots \supseteq B_{m}$, \
$A_{m} \cap B_{1} = \emptyset$, and
$A_{1} \cup B_{m} \not= \emptyset$.

(1)
We use the parallelogram inequality (\ref{paraineqd1d2gen}) in Theorem \ref{THparallelineqgen}.
For 
\[
I = 
 \left\{\begin{array}{ll}
     \{ m, m-2, \ldots, 2 \}  & (\mbox{$m$: even}), \\
     \{ m, m-2, \ldots, 1 \}  & (\mbox{$m$: odd}) ,\\
    \end{array}\right.
\quad
J = 
 \left\{\begin{array}{ll}
     \{ m-1, m-3, \ldots, 1 \}  & (\mbox{$m$: even}), \\
     \{ m-1, m-3, \ldots, 2 \}  & (\mbox{$m$: odd}) ,\\
    \end{array}\right.
\]
we can verify
\begin{align*}
d_{1} 
&= \sum_{i \in I} (\bm{1}_{A_{i}} - \bm{1}_{B_{i}})
= \left\lceil (y-x)/2 \right\rceil
= \left\lceil (x+y)/2 \right\rceil - x,
\\
d_{2} 
&= \sum_{j \in J} (\bm{1}_{A_{j}} - \bm{1}_{B_{j}})
= \left\lfloor (y-x)/2 \right\rfloor
= \left\lfloor (x+y)/2 \right\rfloor -x,
\end{align*}
where
$B_{k} = \emptyset$ for all $k$ by $x \leq y$.
Then the parallelogram inequality (\ref{paraineqd1d2gen}) 
yields the inequality (\ref{discmptconvfn2}).

(2)
For
$I = \{ 2, 4, \ldots, m \}$
and
 $J = \{ 1, 3, \ldots, m -1 \}$,
we have
$d_{1} = \sum_{i \in I} (\bm{1}_{A_{i}} - \bm{1}_{B_{i}})
= \left\lceil (y-x)/2 \right\rceil
= \left\lceil (x+y)/2 \right\rceil -x$
and  
$ d_{2} = \sum_{j \in J} (\bm{1}_{A_{j}} - \bm{1}_{B_{j}})
  = \left\lfloor (y-x)/2 \right\rfloor
  = \left\lfloor (x+y)/2 \right\rfloor -x$.
Then the parallelogram inequality (\ref{paraineqd1d2gen}) 
yields (\ref{discmptconvfn2}).

(3)
By induction on $m =2,3,4,\ldots$, we prove that 
(\ref{discmptconvfn2}) holds 
for all $x, y \in \ZZ\sp{n}$ with $\| x - y \|_{\infty} = m$.
By assumption this is true for $m=2,3$.
Take $x, y \in \ZZ\sp{n}$ with  $\| x - y \|_{\infty} =m$, where $m \geq 4$.
In the representation
$y - x = \sum_{k=1}\sp{m} (\bm{1}_{A_{k}} - \bm{1}_{B_{k}})$
we may assume  $A_{1} \not= \emptyset$; otherwise interchange $x$ and $y$. 
By the parallelogram inequality in Theorem \ref{THparallelineqgen-var}
with $d = \bm{1}_{A_{1}} - \bm{1}_{B_{1}}$
we have $f(x) + f(y)  \geq  f(x + d) +  f(y - d) = f(x') + f(y')$,
where $x'=x+d$ and $y'=y-d$.
Since
$2 \leq \| x' - y' \|_{\infty} \leq m-1$, the induction hypothesis gives
\begin{align*}
 f(x') + f(y')  
 \geq
    f \left(\left\lceil \frac{x'+y'}{2} \right\rceil\right) 
  + f \left(\left\lfloor \frac{x'+y'}{2} \right\rfloor\right) 
=
    f \left(\left\lceil \frac{x+y}{2} \right\rceil\right) 
  + f \left(\left\lfloor \frac{x+y}{2} \right\rfloor\right) .
\end{align*}
Hence (\ref{discmptconvfn2}) holds for $m$.
\end{proof}

In the notation of the Introduction,
Proposition \ref{PRmdptdistant} (2) shows
$\textrm{DMC}(2) = \bigcap_{k=1}\sp{\infty}\textrm{DMC}(2k)$,
and Proposition \ref{PRmdptdistant} (3) shows
$\textrm{DMC}(2) \cap \textrm{DMC}(3) = \mbox{\rm DMC($\geq$2)}$.

\begin{remark} \rm  \label{RMparaineqlocopt}
The local characterization of global minima
for discrete midpoint convex functions,
which is given in Theorem \ref{THintcnvlocopt}
for integrally convex functions,
can be derived easily from the parallelogram inequality 
in Theorem \ref{THparallelineqgen}.
\finbox 
\end{remark}

%%%%%%%%%%%%%% SSSSS %%%%%%%%%%%%%%%%%%%%%
\section{Scaling Operations}
\label{SCscalingDIC}

In designing efficient algorithms 
for discrete or combinatorial optimization,
the proximity-scaling approach is a fundamental technique.
For a function
$f: \mathbb{Z}^{n} \to \mathbb{R} \cup \{ +\infty  \}$
in integer variables and a positive integer
$\alpha$, called a scaling unit,
the $\alpha$-scaling of $f$ means the function $f^{\alpha}$
defined by $f^{\alpha}(x) = f(\alpha x) $ $(x \in \mathbb{Z}^{n})$.
The scaled function  
$f^{\alpha}$ is simpler and hence easier to minimize.
The quality of the obtained minimizer of $f^{\alpha}$ 
as an approximation to the minimizer of $f$
is guaranteed by a proximity theorem,
which states that a (local) minimum 
of the scaled function $f^{\alpha}$ is close to a minimizer 
of the original function $f$.

Discrete midpoint convexity is stable under the scaling operation,
which is the case for sets and functions.
We denote the set of positive integers by $\ZZ_{++}$
and the zero vector by $\bm{0}$.

\begin{proposition} \label{PRdirintcnvSetSC}
Let $S  \subseteq \mathbb{Z}\sp{n}$
 be a discrete midpoint convex set and $\alpha \in \ZZ_{++}$.
Then the scaled set
$S\sp{\alpha} = \{ x \in \ZZ\sp{n} \mid \alpha x \in S  \}$
is discrete midpoint convex.
\end{proposition}
\begin{proof}
Let $S  \subseteq \mathbb{Z}\sp{n}$
be a discrete midpoint convex set.
It suffices to show that, if  $\bm{0}, \alpha x \in S$ and 
$\| x  \|_{\infty} \geq 2$, then
$\alpha \left\lfloor x/2  \right\rfloor,  
\alpha \left\lceil x/2  \right\rceil \in S$.
Since $\bm{0}, \alpha x \in S$ and
$S$ is integrally convex 
by Proposition \ref{PRdirintcnvSetIC}~(2),
we have
$\{ \bm{0}, x, 2x, \ldots, (\alpha -1)x, \alpha x \} \subseteq 
\overline{S} \cap \ZZ\sp{n}= S$. 
The discrete midpoint convexity (\ref{dirintcnvsetdef})
for the pair $(k x, (k+1)x)$ of consecutive points implies
\begin{equation} \label{dicSCprf1a1S}
k x + \left\lfloor  x/2  \right\rfloor,  \  
k x + \left\lceil x/2  \right\rceil  \   \in S
\qquad
(k=0,1,\ldots, \alpha -1).
\end{equation}
Note that $k x$ and $(k+1)x$ 
are at distance $\| k x - (k+1) x \|_{\infty} = \|  x \|_{\infty} \geq 2$,
which allows us to use (\ref{dirintcnvsetdef}).
We will show
\begin{align} 
k x + (\beta +1) \left\lfloor  x/2  \right\rfloor  \   \in S  
\qquad
(k=0,1,\ldots, \alpha -1 - \beta) ,
\label{dicSCprf4SD}
\\
k x + (\beta +1) \left\lceil x/2  \right\rceil  \   \in S
\qquad
(k=0,1,\ldots, \alpha -1 - \beta),
\label{dicSCprf4SU}
\end{align}
respectively,
by induction on $\beta = 0,1, \ldots, \alpha -1$.
For $\beta = \alpha -1$,
we have
$k=0$ from $0 \leq k \leq \alpha -1 - \beta$, and
(\ref{dicSCprf4SD}) and (\ref{dicSCprf4SU})
coincide, respectively, with
$\alpha \left\lfloor x/2  \right\rfloor \in S$ and  
$\alpha \left\lceil x/2  \right\rceil \in S$.
For $\beta = 0$,
(\ref{dicSCprf4SD}) and (\ref{dicSCprf4SU})
are already shown in (\ref{dicSCprf1a1S}).

Assume now that (\ref{dicSCprf4SD}) holds for some
$\beta \geq 0$,
and let 
$k \in \{ 0,1,\ldots, \alpha - 2 - \beta \}$.
Consider consecutive points
$k x + (\beta +1) \left\lfloor  x/2  \right\rfloor$
and $(k+1) x + (\beta +1) \left\lfloor  x/2  \right\rfloor$,
and note that
their $\ell_{\infty}$-distance 
is equal to  $\| x  \|_{\infty}$.
By applying the discrete  midpoint convexity 
(\ref{dirintcnvsetdef})
to this pair of consecutive points
we obtain
\begin{equation*}  %%\label{dicSCprf8S}
 k x + (\beta +1) \left\lfloor  x/2  \right\rfloor 
 + \left\lfloor  x/2  \right\rfloor 
=
  k x + (\beta + 2) \left\lfloor  x/2  \right\rfloor 
  \in S .
\end{equation*}
This shows (\ref{dicSCprf4SD}) for $\beta+1$.
By induction we can conclude that 
(\ref{dicSCprf4SD}) is true for all $\beta$.

The other claim (\ref{dicSCprf4SU}) can be proved in a similar manner.
For a pair of consecutive points
$k x + (\beta +1) \left\lceil  x/2  \right\rceil$
and $(k+1) x + (\beta +1) \left\lceil  x/2  \right\rceil$,
the discrete  midpoint convexity 
(\ref{dirintcnvsetdef}) yields
\begin{equation*}  %%\label{dicSCprf13S}
  k x + (\beta +1) \left\lceil  x/2  \right\rceil
  + \left\lceil x/2  \right\rceil   \ 
=
  k x + (\beta + 2) \left\lceil  x/2  \right\rceil 
  \in S .
\end{equation*}
Then (\ref{dicSCprf4SU}) follows by induction on $\beta$.
\end{proof}

\begin{theorem} \label{THdirintsc}
\quad 

\noindent
{\rm (1)}
Let $f: \mathbb{Z}\sp{n} \to \mathbb{R} \cup \{ +\infty  \}$
be a globally discrete midpoint convex function 
and $\alpha \in \ZZ_{++}$.
Then the scaled function $f\sp{\alpha}$ is globally discrete midpoint convex.

\noindent
{\rm (2)}
Let $f: \mathbb{Z}\sp{n} \to \mathbb{R} \cup \{ +\infty  \}$
be a locally discrete midpoint convex function 
and $\alpha \in \ZZ_{++}$.
Then the scaled function $f\sp{\alpha}$ is locally discrete midpoint convex.
\end{theorem}
\begin{proof}
In either case, $\dom f$ is a discrete midpoint convex set;
this is true by definition in the local case, 
and by Proposition~\ref{PRstrdom} in the global case. 
By Proposition~\ref{PRdirintcnvSetSC},
$\dom f\sp{\alpha}$ is a discrete midpoint convex set.
Then Proposition \ref{PRdirintSC} below establishes (1) and (2). 
\end{proof}

\begin{proposition}  \label{PRdirintSC}
\quad

\noindent
{\rm (1)} 
Let $f$ be a globally discrete midpoint convex function.
If \,  $\bm{0}, \alpha x \in \dom f$ and 
$\| x  \|_{\infty} \geq 2$, then
\begin{equation}\label{dirintcnvfnSC0x}
 f(\bm{0}) + f(\alpha x) \geq
   f \left(\alpha \left\lfloor {x}/{2} \right\rfloor\right) 
  +  f \left(\alpha \left\lceil {x}/{2} \right\rceil\right) .
\end{equation}

\noindent
{\rm (2)} 
Let $f$ be a locally discrete midpoint convex function.
If \, $\bm{0}, \alpha x \in \dom f$ and
$\| x  \|_{\infty} = 2$, then
{\rm (\ref{dirintcnvfnSC0x})} holds.
\end{proposition}
\begin{proof}
The  following proof works for both (1) and (2).
Since $\dom f$ is an integrally convex set
by Proposition \ref{PRdirintcnvSetIC}~(2) with Proposition \ref{PRstrdom},
 we have
\begin{equation} \label{dicSCprf01}
\{ \bm{0}, x, 2x, \ldots, (\alpha -1)x, \alpha x \} \subseteq \dom f .
\end{equation}
The discrete midpoint convexity (\ref{discmptconvfn2})
for the pair $(k x, (k+1)x)$ of consecutive points shows
\begin{equation} \label{dicSCprf1}
 f(k x) + f( (k+1) x) \geq
   f \left(k x + \left\lfloor  x/2  \right\rfloor\right) 
   +   f \left(k x + \left\lceil x/2  \right\rceil\right)
\qquad
(k=0,1,\ldots, \alpha -1). 
\end{equation}
Note that $k x$ and $(k+1)x$ 
are at distance $\| k x - (k+1) x \|_{\infty} = \|  x \|_{\infty}$
and therefore, the inequality (\ref{discmptconvfn2})
can be used in both global and local cases.
The inequality (\ref{dicSCprf1}) implies, in particular, that
\begin{equation} \label{dicSCprf1a1}
k x + \left\lfloor  x/2  \right\rfloor,  \  
k x + \left\lceil x/2  \right\rceil  \   \in \dom f
\qquad
(k=0,1,\ldots, \alpha -1) .
\end{equation}
By adding the inequalities (\ref{dicSCprf1}) for $k=0,1,\ldots, \alpha -1$
we obtain
\begin{equation} \label{dicSCprf3}
 f(\bm{0}) + f(\alpha x) 
 \geq
\sum_{k=0}\sp{\alpha -1}
 \left[
   f \left(k x + \left\lfloor  x/2  \right\rfloor\right) 
   +   f \left(k x + \left\lceil x/2  \right\rceil\right) 
 \right]
 -  2 \sum_{k=1}\sp{\alpha -1} f(k x).
\end{equation}

With a parameter $\beta = 0,1, \ldots, \alpha -1$
we consider a generalized form of this inequality:
\begin{align} 
 f(\bm{0}) + f(\alpha x) 
 \geq &
\sum_{k=0}\sp{\alpha -1 - \beta}
 \left[
   f \left(k x + (\beta +1) \left\lfloor  x/2  \right\rfloor\right) 
   +   f \left(k x + (\beta +1) \left\lceil x/2  \right\rceil\right) 
 \right]
\notag \\ &
 -  \sum_{k=1}\sp{\alpha -1 - \beta} 
 \left[
   f \left(k x + \beta \left\lfloor  x/2  \right\rfloor\right) 
   +   f \left(k x + \beta \left\lceil x/2  \right\rceil\right) 
 \right] .
\label{dicSCprf4}
\end{align}
For $\beta = 0$, 
the inequality (\ref{dicSCprf4}) coincides with
(\ref{dicSCprf3}), 
with each term being finite by (\ref{dicSCprf01}) and (\ref{dicSCprf1a1}).
For $\beta = \alpha -1$,
(\ref{dicSCprf4}) coincides with the desired inequality
(\ref{dirintcnvfnSC0x}).
We will show the inequality (\ref{dicSCprf4})
for $\beta = 0,1,\ldots, \alpha -1$,
together with the finiteness of all the terms.

On the right-hand side of (\ref{dicSCprf4}),
we classify terms into two parts,
$D(\beta)$ and $U(\beta)$,
according to whether they contain
$\left\lfloor  x/2  \right\rfloor$
or $\left\lceil x/2  \right\rceil$, i.e.,
\begin{align} 
 D(\beta) & =
\sum_{k=0}\sp{\alpha -1 - \beta}
   f \left(k x + (\beta +1) \left\lfloor  x/2  \right\rfloor \right) 
-  \sum_{k=1}\sp{\alpha -1 - \beta} 
   f \left(k x + \beta \left\lfloor  x/2  \right\rfloor \right) ,
\label{dicSCprf6}
 \\ 
 U(\beta) &=
\sum_{k=0}\sp{\alpha -1 - \beta}
  f \left(k x + (\beta +1) \left\lceil x/2  \right\rceil \right) 
-  \sum_{k=1}\sp{\alpha -1 - \beta} 
  f \left(k x + \beta \left\lceil x/2  \right\rceil \right) . 
\label{dicSCprf7}
\end{align}
Then (\ref{dicSCprf4}) is expressed as
\begin{equation} \label{dicSCprf5}
 f(\bm{0}) + f(\alpha x) 
 \geq  D(\beta) + U(\beta)
\qquad (\beta = 0,1, \ldots, \alpha -1) .
\end{equation}
This inequality for $\beta =0$ is shown in (\ref{dicSCprf3}).
The general case with $\beta \geq 1$ 
follows from Lemma \ref{LMdicbeta} below.
\end{proof}

\begin{lemma}  \label{LMdicbeta}
\quad

\noindent
{\rm (1)}
$D(\beta)  \geq  D(\beta +1)$
\quad
$(\beta = 0,1, \ldots, \alpha -2)$.

\noindent
{\rm (2)}
$U(\beta) \geq  U(\beta +1)$
\quad
$(\beta = 0,1, \ldots, \alpha -2)$.
\end{lemma}

\begin{proof}
(1)
Consider consecutive terms in the first summation 
on the right-hand side of (\ref{dicSCprf6})
and note that
the $\ell_{\infty}$-distance between
$(k x + (\beta +1) \left\lfloor  x/2  \right\rfloor)$
and $((k+1) x + (\beta +1) \left\lfloor  x/2  \right\rfloor )$
is equal to  $\| x  \|_{\infty}$.
By applying the discrete  midpoint convexity 
(\ref{discmptconvfn2})
to this pair of consecutive terms and using the identity
$\left\lfloor  x/2  \right\rfloor + \left\lceil x/2  \right\rceil = x$,
we obtain
\begin{align} 
 &  f \left(k x + (\beta +1) \left\lfloor  x/2  \right\rfloor \right) 
  +   f \left((k+1) x + (\beta +1) \left\lfloor  x/2  \right\rfloor \right) 
\notag \\ & \geq
  f \left(k x + (\beta +1) \left\lfloor  x/2  \right\rfloor 
 + \left\lfloor  x/2  \right\rfloor \right) 
  +   f \left(k x + (\beta +1) \left\lfloor  x/2  \right\rfloor
  + \left\lceil x/2  \right\rceil \right) 
\notag \\ & =
  f \left(k x + (\beta + 2) \left\lfloor  x/2  \right\rfloor  \right) 
  +   f \left((k+1) x + \beta \left\lfloor  x/2  \right\rfloor \right) 
\qquad (k=0,1,\ldots, \alpha - 2 - \beta).
\label{dicSCprf8}
\end{align}
We add (\ref{dicSCprf8}) for 
$k=0,1,\ldots, \alpha - 2 - \beta$.
The sum of the left-hand sides is given by
\begin{align} 
 & 
\sum_{k=0}\sp{\alpha -2 - \beta}
 f \left(k x + (\beta +1) \left\lfloor  x/2  \right\rfloor \right) 
  + 
\sum_{k=0}\sp{\alpha - 2 - \beta}
  f \left((k+1) x + (\beta +1) \left\lfloor  x/2  \right\rfloor \right) 
\notag \\ & =
\sum_{k=0}\sp{\alpha -1 - \beta}
 f \left(k x + (\beta +1) \left\lfloor  x/2  \right\rfloor \right) 
  + 
\sum_{k=1}\sp{\alpha - 2 - \beta}
  f \left(k x + (\beta +1) \left\lfloor  x/2  \right\rfloor \right) 
\label{dicSCprf9}
\end{align}
and the sum of the right-hand sides is given by
\begin{align} 
& 
\sum_{k=0}\sp{\alpha -2 - \beta}
  f \left(k x + (\beta + 2) \left\lfloor  x/2  \right\rfloor  \right) 
  +
\sum_{k=0}\sp{\alpha -2 - \beta}
   f \left((k+1) x + \beta \left\lfloor  x/2  \right\rfloor \right) 
\notag \\ & =
\sum_{k=0}\sp{\alpha -2 - \beta}
  f \left(k x + (\beta + 2) \left\lfloor  x/2  \right\rfloor  \right) 
  +
\sum_{k=1}\sp{\alpha -1 - \beta}
   f \left( k x + \beta \left\lfloor  x/2  \right\rfloor \right) .
\label{dicSCprf10}
\end{align}
Then it follows from $\mbox{(\ref{dicSCprf9})} \geq \mbox{(\ref{dicSCprf10})}$
that
\begin{align*} 
& \sum_{k=0}\sp{\alpha -1 - \beta}
 f \left(k x + (\beta +1) \left\lfloor  x/2  \right\rfloor \right) 
  + 
\sum_{k=1}\sp{\alpha - 2 - \beta}
  f \left(k x + (\beta +1) \left\lfloor  x/2  \right\rfloor \right) 
\notag \\ & \geq 
\sum_{k=0}\sp{\alpha -2 - \beta}
  f \left(k x + (\beta + 2) \left\lfloor  x/2  \right\rfloor  \right) 
  +
\sum_{k=1}\sp{\alpha -1 - \beta}
   f \left( k x + \beta \left\lfloor  x/2  \right\rfloor \right) ,
\end{align*}
which is equivalent to 
\begin{align*} 
& \sum_{k=0}\sp{\alpha -1 - \beta}
 f \left(k x + (\beta +1) \left\lfloor  x/2  \right\rfloor \right) 
- \sum_{k=1}\sp{\alpha -1 - \beta}
   f \left( k x + \beta \left\lfloor  x/2  \right\rfloor \right) 
\notag \\ & \geq 
\sum_{k=0}\sp{\alpha -2 - \beta}
  f \left(k x + (\beta + 2) \left\lfloor  x/2  \right\rfloor  \right) 
- \sum_{k=1}\sp{\alpha - 2 - \beta}
  f \left(k x + (\beta +1) \left\lfloor  x/2  \right\rfloor \right) .
\end{align*}
The left-hand side above is equal to $D(\beta)$
and the right-hand side is $D(\beta + 1)$.
We have thus shown
$D(\beta)  \geq  D(\beta +1)$.

(2) The proof is similar to that of case (1).
In place of (\ref{dicSCprf8}) we now have
\begin{align*} 
 &  f \left(k x + (\beta +1) \left\lceil  x/2  \right\rceil \right) 
  +   f \left((k+1) x + (\beta +1) \left\lceil  x/2  \right\rceil \right) 
\notag \\ & \geq
  f \left(k x + (\beta +1) \left\lceil  x/2  \right\rceil 
  + \left\lceil x/2  \right\rceil \right) 
  +   f \left(k x + (\beta +1) \left\lceil  x/2  \right\rceil
 + \left\lfloor  x/2  \right\rfloor \right) 
\notag \\ & =
  f \left(k x + (\beta + 2) \left\lceil  x/2  \right\rceil  \right) 
  +   f \left((k+1) x + \beta \left\lceil  x/2  \right\rceil \right) 
\qquad (k=0,1,\ldots, \alpha - 2 - \beta).
\end{align*}
In the remaining part, we replace
$\left\lfloor  x/2  \right\rfloor$ by
$\left\lceil x/2  \right\rceil$.
\end{proof}

Theorem \ref{THdirintsc} is analogous to 
the well-known fact \cite{Mdcasiam} that
${\rm L}^{\natural}$-convexity is stable under scaling.
In contrast, integral convexity admits the scaling operation only when $n \leq 2$,
while for $n \geq 3$, the scaled function $f^{\alpha}$  is 
not always integrally convex \cite{MMTT16proxICissac,MMTT17proxIC}.

%%%%%%%%%%%%%% SSSSS %%%%%%%%%%%%%%%%%%%%%
\section{Proximity Theorem}
\label{SCdirintcnvprox}

Let $f: \mathbb{Z}^{n} \to \mathbb{R} \cup \{ +\infty  \}$ 
and $\alpha \in \ZZ_{++}$.
We say that $x^{\alpha} \in \dom f$ is an {\em $\alpha$-local minimizer} of $f$
(or {\em $\alpha$-local minimal} for $f$)
if $f(x^{\alpha}) \leq f(x^{\alpha}+ \alpha d)$
for all $d \in  \{ -1,0, +1 \}^{n}$.
In general terms a proximity theorem states that
for $\alpha \in \ZZ_{++}$ there exists an integer
$B(n,\alpha) \in \ZZ_{+}$ such that
if $x^{\alpha}$ is an $\alpha$-local minimizer of $f$,
then there exists a minimizer $x^{*}$ of $f$ satisfying 
$ \| x^{\alpha} - x^{*}\|_{\infty} \leq B(n,\alpha)$,
where $B(n,\alpha)$ is called the {\em proximity distance}.

The following proximity theorem holds for discrete midpoint convex functions.
It states that a local minimizer $x\sp{\alpha}$ 
of the scaled function $f^{\alpha}(x) = f(\alpha x)$ 
is close to a minimizer of the original 
discrete midpoint convex function $f(x)$. 
Note that $f^{\alpha}$ is discrete midpoint convex
by Theorem \ref{THdirintsc},
and $x\sp{\alpha}$
is in fact a global minimizer of $f^{\alpha}$ 
by Theorems \ref{THintcnvlocopt} and \ref{THdirintclass}.

\begin{theorem}  \label{THdirintprox}
Let $f: \mathbb{Z}\sp{n} \to \mathbb{R} \cup \{ +\infty  \}$
be a (globally or locally) discrete midpoint convex function,
$\alpha \in \mathbb{Z}_{++}$, 
and $x\sp{\alpha} \in \dom f$.
If $f(x\sp{\alpha}) \leq f(x\sp{\alpha}+ \alpha d)$ for all 
$d \in  \{ -1,0, +1 \}\sp{n}$,
then there exists a minimizer 
$x\sp{*} \in \ZZ\sp{n}$ of $f$ with
$  \| x\sp{\alpha} - x\sp{*}  \|_{\infty} \leq n (\alpha - 1)$.
\finbox
\end{theorem}

To prove Theorem~\ref{THdirintprox}
we may assume $x^{\alpha} = \bm{0}$.
Define 
$S = 
  \{ x \in \ZZ^{n} \mid \| x \|_{\infty} \leq n(\alpha - 1) \}$
and
$ W =   \{ x \in \ZZ^{n} \mid \| x \|_{\infty} = n(\alpha - 1) +1 \}$,
and 
let $\mu$ be the minimum of $f(x)$ taken over $x \in S$
and $\hat x$ be a point in $S$ with $f(\hat x)=\mu$.
We shall show
\begin{equation} \label{prfT1al3mu4}
  f(y) \geq \mu 
\qquad
\mbox{for all $y \in W$} .
\end{equation}
Then Theorem~\ref{THintcnvbox} (box-barrier property) implies that 
$f(z) \geq \mu$ for all $z \in \ZZ^{n}$.

Fix $y =(y_{1}, \ldots y_{n}) \in W$
and let $m = \| y \|_{\infty}$, which is equal to $n (\alpha - 1)+1$.
With
\[
 A_{k} = \{ i \mid y_{i} \geq m+1-k \},
\quad
 B_{k} = \{ i \mid y_{i} \leq -k \} 
\qquad (k=1,\ldots,m), 
\]
we can represent $y$ as
\begin{equation} \label{yrepLovpmdimna}
y = \sum_{k=1}^{m} (\bm{1}_{A_{k}} - \bm{1}_{B_{k}}) .
\end{equation}
We have
$A_{1} \subseteq A_{2}  \subseteq \cdots \subseteq A_{m}$, \
$B_{1} \supseteq B_{2} \supseteq \cdots \supseteq B_{m}$, \
$A_{m} \cap B_{1} = \emptyset$,
and $A_{1} \cup B_{m} \not= \emptyset$.

\medskip

\begin{lemma} \label{LMyrepmultalpha}
There exists some $k_{0} \in \{ 1,2,\ldots, m - \alpha +1 \}$  
such that
$(A_{k_{0}}, B_{k_{0}}) =(A_{k_{0}+j}, B_{k_{0}+j})$ for 
$j=1,2,\ldots, \alpha -1$.
\end{lemma}
\begin{proof}
We may assume $A_{1} \not= \emptyset$ by $A_{1} \cup B_{m} \not= \emptyset$
and Propositions \ref{PRdirintcnvinvarS} (3) and \ref{PRdirintcnvinvarW}.
Define $(a_{k}, b_{k}) = (|A_{k}|, n-|B_{k}|)$ 
for $k=1,2,\ldots, m$ and 
$s=| \{ i \mid y_{i} \geq 1 \}|$.
%%$s=|\suppp(y)|$.
The sequence $\{ (a_{k}, b_{k}) \}_{k=1,2,\ldots, m}$ is
nondecreasing in $\ZZ^{2}$,
satisfying
$1 \leq a_{1} \leq a_{2} \leq \cdots \leq a_{m} \leq s$
and
$s \leq b_{1} \leq b_{2}\leq \cdots \leq  b_{m} \leq n$.
That is,
$(1,s) \leq (a_{1},b_{1}) \leq (a_{2},b_{2}) \leq \cdots \leq (a_{m}, b_{m}) \leq (s,n)$.
Since $m = n(\alpha - 1)+1$ and
the length of a strictly increasing chain 
connecting $(1,s)$ to $(s,n)$ 
in $\ZZ^{2}$ is bounded by $n$,
the sequence $\{ (a_{k}, b_{k}) \}_{k=1,2,\ldots, m}$
must contain,
by the pigeonhole principle, 
 a constant subsequence of length $\geq \alpha$.
Hence follows the claim.
\end{proof}

With reference to the index $k_{0}$ in Lemma \ref{LMyrepmultalpha}
we define a bipartition $(I,J)$ of $\{ 1,2,\ldots, m \}$
by 
$I = \{ 1, 2, \ldots, k_{0}-1  \} 
\cup \{ k_{0}+ \alpha, k_{0}+ \alpha +1, \ldots, m \} $
and $J = \{ k_{0}, k_{0}+1, \ldots, k_{0}+ \alpha -1 \} $.
By the parallelogram inequality (\ref{paraineqd1d2gen}) 
in Theorem \ref{THparallelineqgen},
where 
$d_{1} = \sum_{i \in I} (\bm{1}_{A_{i}} - \bm{1}_{B_{i}})$
and  
$ d_{2} = \sum_{j \in J} (\bm{1}_{A_{j}} - \bm{1}_{B_{j}})= \alpha d_{0}$
with
$d_{0} =  \bm{1}_{A_{k_{0}}} - \bm{1}_{B_{k_{0}}}$,
we obtain
\[ 
 f(\bm{0}) + f(y) \geq f(d_{1}) + f(\alpha d_{0}) .
\] 
Here we have
$f(\alpha d_{0}) \geq  f(\bm{0})$
by the $\alpha$-local minimality of $\bm{0}$. 
We also have $d_{1}\in S$ since 
\[
\| d_{1} \|_{\infty}
 = \| \sum_{i \in I} (\bm{1}_{A_{i}} - \bm{1}_{B_{i}}) \|_{\infty} 
 = |I| = m- \alpha  = (n-1)(\alpha - 1) \leq n(\alpha - 1),
\]
and hence
$f(d_{1}) \geq  \mu$
by the definition of $\mu$.
Therefore,
\[
f(y) \geq  f(d_{1})  + [f(\alpha d_{0}) -  f(\bm{0}) ]
\geq  \mu + 0 = \mu .
\]
This establishes (\ref{prfT1al3mu4}), 
completing the proof of Theorem~\ref{THdirintprox}.

Theorem~\ref{THdirintprox} is a generalization of
the proximity theorem for ${\rm L}^{\natural}$-convex functions below.
We have the same proximity bound $n (\alpha -1)$ in both theorems.

\begin{theorem}[{\cite{IS03,Mdcasiam}}] \label{THlfnproximity}
Let $f: \mathbb{Z}\sp{n} \to \mathbb{R} \cup \{ +\infty  \}$
be an ${\rm L}^{\natural}$-convex function,
$\alpha \in \mathbb{Z}_{++}$, and $x\sp{\alpha} \in \dom f$.
If $f(x^{\alpha}) \leq f(x^{\alpha} +  \alpha d)$ for all 
$d \in  \{ 0, 1 \}^{n} \cup \{ 0, -1 \}^{n}$,
then there exists a minimizer 
$x^{*}$ of $f$ with $\| x^{\alpha} - x^{*}  \|_{\infty} \leq n (\alpha -1)$.
\finbox
\end{theorem}

The following example
demonstrates the tightness of the bound in Theorems \ref{THdirintprox} and \ref{THlfnproximity}.
This example also shows that the bound $n (\alpha -1)$ is tight
for a linear function defined on a simple polyhedron.

\begin{example}[\cite{MMTT16proxICissac,MMTT17proxIC,MT02proxRIMS}]\rm \label{EXproxtigntLnat}
Let $X = \{ x \in \ZZ^{n} \mid 
0 \leq x_{i} - x_{i+1} \leq \alpha -1 \ (i=1,\ldots,n-1), \ 
0 \leq x_{n} \leq \alpha -1  \}$.
The function $f$ defined by $f(x)=-x_{1}$ on $\dom f =X$ is
an ${\rm L}^{\natural}$-convex 
(and hence discrete midpoint convex) function
and  has a unique minimizer at 
$x^{*} =  (n (\alpha -1),(n-1)(\alpha -1),\ldots,2(\alpha -1),\alpha -1)$.
On the other hand, $x^{\alpha} = \veczero$ is  $\alpha$-local minimal,
since
$X \cap \{ -\alpha, 0,\alpha \}^{n} = \{ \bm{0} \}$.
We have $ \| x^{\alpha} - x^{*}\|_{\infty} = n (\alpha - 1)$,
which shows the tightness of the bound $n (\alpha - 1)$ 
given in Theorem~\ref{THlfnproximity}.
This example is given in \cite{MMTT16proxICissac,MMTT17proxIC}
as a reformulation of \cite[Remark~2.3]{MT02proxRIMS} 
for {\rm L}-convex functions to ${\rm L}^{\natural}$-convex functions.
\finbox
\end{example}

For (general) integrally convex functions, in contrast, 
the proximity bound $n (\alpha -1)$ is valid only when $n \leq 2$,
and a quadratic lower bound 
$(n-2)^{2}(\alpha -1)/4$
for the proximity distance is demonstrated in \cite{MMTT16proxICissac,MMTT17proxIC}
(no better lower bound is known).
The following proximity theorem,
with a superexponential proximity bound, is known for integrally convex functions.

\begin{theorem}[\cite{MMTT16proxICissac,MMTT17proxIC}]  \label{THproxintcnv}
Let $f: \mathbb{Z}^{n} \to \mathbb{R} \cup \{ +\infty  \}$
be an integrally convex function,
$\alpha \in \mathbb{Z}_{++}$, 
and $x^{\alpha} \in \dom f$.
If 
$ f(x^{\alpha}) \leq f(x^{\alpha}+ \alpha d)$
for all $ d \in  \{ -1,0, +1 \}^{n}$,
then there exists a minimizer 
$x^{*}$ of $f$ with 
$ \| x^{\alpha} - x^{*}\|_{\infty} \leq \beta_{n} (\alpha - 1)$,
where 
$\beta_{1}=1$, $\beta_{2}=2$, and
$ \beta_{n} \leq {(n+1)!}/{2^{n-1}}$
$(n=3,4,\ldots)$.
\finbox
\end{theorem}

%%%%%%%%%%%%%% SSSSS %%%%%%%%%%%%%%%%%%%%%

\section{Minimization Algorithms}
\label{SCminalg}

\subsection{2-neighborhood steepest descent algorithm}
\label{SCdescalg}

In this section we propose a variant of the descent algorithm
that is suited for discrete midpoint convex functions.
The algorithm repeats finding a local minimizer in the neighborhood of $\ell_{\infty}$-distance 2,
and is named the {\em 2-neighborhood steepest descent algorithm}.

Let
$f: \ZZ\sp{n} \to \RR \cup \{ +\infty \}$ 
be a (locally or globally) discrete midpoint convex function with
$\argmin f \neq \emptyset$.
Fix any 
\[
  x\sp{(0)} \in \dom f \setminus \argmin f 
\]
and denote the minimum $\ell_{\infty}$-distance to a minimizer of $f$ by  
\begin{equation}  \label{algLdef}
  L = L(x\sp{(0)}) =\min\{ \|x - x\sp{(0)} \|_{\infty} \mid x \in \argmin f\}.
\end{equation}
For $k = 2,3,\ldots,L$, define
\begin{equation}  \label{algSkdef}
S_{k} = S_{k}(x\sp{(0)}) = \{ x \in \mathbb{Z}^{n} \mid \| x - x\sp{(0)} \|_{\infty} \leq k \} .
\end{equation}

The idea of our algorithm is to generate a sequence of points that minimize
$f$ within $S_{k}$ for $k=2,3, \ldots, L$.
The following proposition states that the consecutive points 
can be chosen to be close to each other, at $\ell_{\infty}$-distance less than or equal to 2.

\begin{proposition}\label{PR2nbhmin}
Assume $3 \leq k \leq L$.
For any $x\sp{(k-1)} \in \argmin\{ f(x) \mid x \in S_{k-1} \}$,
there exists 
$x\sp{(k)} \in \argmin\{ f(x) \mid x \in S_{k} \}$
that satisfies
$\| x\sp{(k)} - x\sp{(k-1)} \|_{\infty} \leq 2$.
\end{proposition}
\begin{proof}
Let $y$ be an element of
$\argmin\{ f(x) \mid x \in S_{k}\}$
with $\| y - x\sp{(k-1)} \|_{\infty}$ minimum,
and put $m = \| y - x\sp{(k-1)} \|_{\infty}$.
We can prove $m \leq 2$ (see below). Then we can take $x\sp{(k)} = y$.

To prove  $m \leq 2$ by contradiction, suppose that $m \geq 3$.
We consider the decomposition of
$y- x\sp{(k-1)}$ in the form of (\ref{DICdecAkBksum}) with 
$ A_{j} = \{ i \mid y_{i}- x\sp{(k-1)}_{i} \geq m+1-j \}$
and
$ B_{j} = \{ i \mid y_{i}- x\sp{(k-1)}_{i} \leq -j \} $
for $j = 1,2,\ldots,m$.
Let $d = \bm{1}_{A_{2}} - \bm{1}_{B_{2}}$.
By the parallelogram inequality (\ref{paraineqxdyd}) in Theorem~\ref{THparallelineqgen-var},
we obtain
\begin{align} \label{paraineqxk1dyd}
 f(x\sp{(k-1)}) + f(y) \geq f(x\sp{(k-1)}+d) + f(y-d).
\end{align}
Here we have
\[
x\sp{(k-1)}+d \in S_{k-1},
 \qquad  
y-d \in S_{k}, 
\]
which are shown in Claim below.
Since $x\sp{(k-1)}$ is a minimizer in $S_{k-1}$,
the former implies 
$f(x\sp{(k-1)}+d) \geq f(x\sp{(k-1)})$.
The latter implies the strict inequality
$f(y-d) > f(y)$ by 
$\| (y-d) - x\sp{(k-1)}\|_{\infty} < \| y - x\sp{(k-1)}\|_{\infty}$
and the definition of $y$.
The addition of these two inequalities yields a contradiction to (\ref{paraineqxk1dyd}).

\bigskip

\noindent
{\bf Claim.} 

(i) $(x\sp{(k-1)}_{i}+1) -  x\sp{(0)}_{i} \leq  k-1$ 
for $i \in A_2$.

(ii) $(x\sp{(k-1)}_{i} -1) - x\sp{(0)}_{i} \geq  - (k-1)$ 
for $i \in B_2$.

(iii) $(y_{i} -1) - x\sp{(0)}_{i} \geq  - k$ 
for  $i \in A_2$.

(iv) $(y_{i} +1) -  x\sp{(0)}_{i}  \leq k$  
for  $i \in B_2$.

\noindent
(Proof)
(i)
By $y \in S_{k}$, $i \in A_2$, and $m \geq 3$ we have
$ x\sp{(0)}_{i} + k \geq
 y_{i} \geq x\sp{(k-1)}_{i} + (m - 1) \geq x\sp{(k-1)}_{i} + 2$.
\\
(ii)
By $y \in S_{k}$ and $i \in B_2$ we have
$ x\sp{(0)}_{i} -k \leq  y_{i} \leq x\sp{(k-1)}_{i} - 2$.
\\
(iii)
By $x\sp{(k-1)} \in S_{k-1}$, $i \in A_2$, and $m \geq 3$ we have
$ x\sp{(0)}_{i} - (k-1) \leq  x\sp{(k-1)}_{i} \leq y_{i} - (m - 1) \leq y_{i} - 2$.
\\
(iv) 
By $x\sp{(k-1)} \in S_{k-1}$ and $i \in B_2$ we have
$ x\sp{(0)}_{i} + (k-1) \geq  x\sp{(k-1)}_{i} \geq  y_{i} + 2$.
\end{proof}

On the basis of Proposition \ref{PR2nbhmin}
we can devise an algorithm 
which generates a sequence of points 
$\{ x\sp{(k)} \}_{k}$ 
such that
$x\sp{(k)}$ is a minimizer of $f$ in $S_{k}(x\sp{(0)})$
and is at $\ell_{\infty}$-distance at most two from the preceding point $x\sp{(k-1)}$.
For $x \in \ZZ\sp{n}$ we define 
\[
 N_{2}(x) = \{ y  \in \ZZ\sp{n} \mid \|y - x\|_{\infty} \leq 2 \}
\]
and call it the {\em 2-neighborhood} of $x$.
We assume that an oracle is available which
finds, for any $x \in \dom f$, a minimizer of $f$
in $ N_{2}(x)$ intersected with a box.
We refer to such an oracle as a {\em 2-neighborhood minimization oracle}.
We also assume that an initial point $x\sp{(0)}$ in $\dom f$ can be found (or is given).

%%%%%%%%%%%%%%%%%
\begin{tabbing}     
\= {\bf The 2-neighborhood steepest descent algorithm}%
\\
\> \quad  D0: 
   \= Find  $x\sp{(0)} \in \dom f$, 
     find $x\sp{(2)}$  that minimizes $f(x)$ in $S_{2}(x\sp{(0)})$,
     and set $k:=3$.
\\
\> \quad  D1:
   \>  Find $x\sp{(k)}$  that minimizes $f(x)$ in $N_{2}(x\sp{(k-1)}) \cap S_{k}(x\sp{(0)})$.
\\
\> \quad  D2: 
    \> If $f(x\sp{(k)}) = f(x\sp{(k-1)})$, then 
        output $x\sp{(k-1)}$ and  stop.             
\\
\> \quad  D3: 
  \> Set  $k := k+1$, and go to D1.  
\end{tabbing}
%%%%%%%%%%%%%

The following theorem states that this algorithm finds a minimizer of $f$ in $L$ iterations.
It is emphasized that we do not have to know $L$ in advance, but 
we can obtain $L$ as an outcome of the algorithm.

\begin{theorem}\label{THalg2nbdSD}
For a (globally or locally) midpoint convex function $f$  with
$\argmin f \not=\emptyset$,
the 2-neighborhood steepest descent algorithm
finds a minimizer of $f$ in $L$ iterations.
\end{theorem}
\begin{proof}
By Proposition \ref{PR2nbhmin} the point 
$x\sp{(k)}$ generated by the algorithm is
in fact a minimizer of $f$ in $S_{k}(x\sp{(0)})$.
That is, we have
\begin{equation}\label{2NSD-eq1}
  x\sp{(k)} \in \argmin \{f(x) \mid x \in S_{k}(x\sp{(0)}) \} 
\qquad (k=2,3,\ldots).
\end{equation}
{\bf Claim.}
If $f(x\sp{(k)}) = f(x\sp{(k-1)})$ in Step D2,
then $x\sp{(k-1)}$ is a minimizer of $f$.
\\
(Proof)
For any $d \in  \{ -1, 0, +1 \}^{n}$,
$x\sp{(k-1)} + d$ belongs to $S_{k}(x\sp{(0)})$, and hence
$f(x\sp{(k)}) \leq f(x\sp{(k-1)} +  d)$ by (\ref{2NSD-eq1}).
Therefore, if $f(x\sp{(k)}) = f(x\sp{(k-1)})$,
then $f(x\sp{(k-1)}) \leq f(x\sp{(k-1)} +  d)$ for any $d$.
This implies $x\sp{(k-1)} \in \argmin f$
by Theorem \ref{THintcnvlocopt},
since $f$ is integrally convex by Theorem \ref{THdirintclass}.
(End of the proof of Claim)

\medskip

By the definition (\ref{algLdef}) of $L$,
$S_{k}(x\sp{(0)})$
contains a minimizer of $f$ if and only if $k \geq L$.
In particular, 
$x\sp{(L)}$ is a minimizer of $f$ and  $f(x\sp{(L)}) = f(x\sp{(L+1)})$,
whereas 
$f(x\sp{(k)}) \not= f(x\sp{(k-1)})$ for any $k \leq L$.
Therefore, the number of iterations of the above algorithm is equal to $L$.
\end{proof}

Let $F(n)$ denote the number of function evaluations 
needed by the 2-neighborhood minimization oracle.
Since $N_{2}(x)$ consists of $5\sp{n}$ points,
$F(n)$ is bounded by $5\sp{n}$.
However, it can be smaller for a subclass of discrete midpoint convex functions;
for example,
$F(n)$ is polynomial in $n$ for ${\rm L}^{\natural}$-convex functions.
It is also noted that $F(n)$ cannot be polynomial in $n$
for general discrete midpoint convex functions,
since every function on the unit cube $\{0, 1 \}^n$ is trivially discrete midpoint convex. 
By Theorem \ref{THalg2nbdSD}, 
the total number of function evaluations of the above algorithm 
is ${\rm O}(F(n) L)$, which is bounded by ${\rm O}(5^{n} L)$.

\begin{remark} \rm  \label{RMkolmshiomuro}
Theorem \ref{THalg2nbdSD} shows that the sequence of points
generated by the 2-neighborhood steepest descent algorithm
connects the initial point $x\sp{(0)}$ and a nearest minimizer of $f$ 
with $L-1$ intermediate points,
where $L$ is the minimum $\ell_{\infty}$-distance 
from the initial point $x\sp{(0)}$ to a minimizer of $f$.
Similar facts are pointed out for 
${\rm L}^{\natural}$-convex function minimization algorithms 
by Kolmogorov and Shioura \cite{KS09lnatmin}
and Murota and Shioura \cite{MS14exbndLmin}; 
see also Shioura \cite{Shi17L}.
\finbox 
\end{remark}

\subsection{Scaling algorithm}
\label{SCscalingalg}

The scaling property (Theorem \ref{THdirintsc}) and the proximity theorem (Theorem \ref{THdirintprox})
enable us to design a scaling algorithm for the minimization of 
(globally or locally) discrete midpoint convex functions
with bounded effective domains.
The proposed algorithm employs the standard proximity-scaling framework
while using the 2-neighborhood steepest descent algorithm of Section \ref{SCdescalg} as a subroutine.

For each scaling unit $\alpha$, we minimize the scaled function $f\sp{\alpha}(y) = f(x + \alpha y)$
with the origin at the current $x$,
find a minimizer $y$ of $f\sp{\alpha}(y)$ 
by the 2-neighborhood steepest descent algorithm of Section \ref{SCdescalg},
and update $x$ to $x+ \alpha y$. 
It is emphasized that the scaled function $f\sp{\alpha}$ is discrete midpoint convex 
by Theorem \ref{THdirintsc}
and it has a minimizer $y$ with the magnitude $\| y \|_{\infty}$ 
controlled by the proximity result in Theorem \ref{THdirintprox}.
Recall that $K_{\infty}$
denotes the $\ell_{\infty}$-size of the effective domain of $f$, i.e.,
$K_{\infty} = \max\{ \| x -y \|_{\infty} \mid x, y \in \dom f \}$,
where we assume $K_{\infty} >0$ to avoid triviality.

%%%%%%%%%%%%%%%%%
\begin{tabbing}     
\= {\bf Scaling algorithm for discrete midpoint convex functions}%
\\
\> \quad  S0: 
   \= Find a vector $x \in \dom f$ and set 
   $\alpha := 2\sp{\lceil \log_{2} (K_{\infty} +1) \rceil}$.
\\
\> \quad  S1:
   \>  Find a vector $y$  that minimizes  
      $f\sp{\alpha}(y) =   f(x + \alpha y)$
      subject to the constraint $\| y \|_{\infty} \leq n$
 \\ \> \>   (by the 2-neighborhood steepest descent algorithm),   
   and set $x:= x+ \alpha y$.  \\
\> \quad  S2: 
    \> If $\alpha = 1$, then stop \
       ($x$ is a minimizer of $f$).             
\\
\> \quad  S3: 
  \> Set  $\alpha:=\alpha/2$, and go to S1.  
\end{tabbing}
%%%%%%%%%%%%%

The correctness of the algorithm can be shown as follows.
Let $x\sp{2\alpha}$ denote 
the vector $x$ at the beginning of Step~S1 for $\alpha$.
The function 
$f\sp{\alpha}(y) =  f(x\sp{2\alpha} + \alpha y)$
is discrete midpoint convex by Theorem \ref{THdirintsc}.
Let $y\sp{\alpha}$ be the vector computed in Step~S1
and 
$x\sp{\alpha} = x\sp{2\alpha} + \alpha y\sp{\alpha}$.
The proximity theorem (Theorem \ref{THdirintprox}) applied to $f\sp{\alpha}$ guarantees that 
if $y=\bm{0}$ is 2-local minimal for $f\sp{\alpha}$, then $y\sp{\alpha}$ is 
a (global) minimizer of $f\sp{\alpha}$.
In other words,
if $x\sp{2\alpha}$ is $2\alpha$-local minimal for $f$, 
then $x\sp{\alpha}$ is $\alpha$-local minimal for $f$.
At the beginning of the algorithm we have
$\alpha = 2\sp{\lceil \log_{2} (K_{\infty} +1) \rceil}$,
and hence
the initial vector $x=x\sp{2\alpha}$ is obviously $2\alpha$-local minimal for $f$.
At the termination of the algorithm we have $\alpha = 1$, and 
the output of the algorithm is indeed a minimizer of $f$,
since $1$-local minimality for discrete midpoint convex functions 
is the same as the global minimality 
by Theorems \ref{THintcnvlocopt} and \ref{THdirintclass}.

The complexity of the algorithm can be analyzed as follows.
By Theorem \ref{THalg2nbdSD},
Step~S1 can be done with at most $n$ calls of the 2-neighborhood minimization oracle.
Since the number of scaling phases is $\log_{2} K_{\infty}$,
the total number of  oracle calls 
is ${\rm O}(n \log_{2} K_{\infty})$. 
As for Step~S0, we assume that 
an initial point $x$ and the size $K_{\infty}$ are available,
since we have no general efficient method for computing them.
With this understanding we obtain the following theorem.
Recall that $F(n)$ denotes the number of function evaluations 
needed by the 2-neighborhood minimization oracle.

\begin{theorem} \label{THscalgMPC}
For a (globally or locally) midpoint convex function $f$  with
a bounded effective domain $\dom f$,
the scaling algorithm finds a minimizer of $f$ with 
${\rm O}(n F(n) \log_{2} K_{\infty})$ function evaluations.
\finbox
\end{theorem}

The algorithm runs with
${\rm O}(C(n) \log_{2} K_{\infty})$ 
function evaluations for $C(n)=n F(n)$.
This means that, if the dimension $n$ is fixed,
the algorithm is polynomial in the problem size.
With the naive bound $F(n) \leq 5^{n}$,
we obtain $C(n) \leq  5^{n} n$
and the complexity of the scaling algorithm above is
${\rm O}(5^{n} n \log_{2} K_{\infty})$.
It is noted that no algorithm can be polynomial in $n$
for general discrete midpoint convex functions,
since every function on the unit cube $\{0, 1 \}^n$ 
is trivially discrete midpoint convex.

Our algorithm has much less complexity than 
the scaling algorithm for general integrally convex functions developed in \cite{MMTT17proxIC}.
The number of function evaluations of the algorithm of \cite{MMTT17proxIC}
is known to be bounded by
${\rm O}(C_{0}(n) \log_{2} K_{\infty})$ 
with $C_{0}(n) = (12(n+1)!/2^{n-1})^{n}$,
which is much larger than 
$C(n) \leq 5^{n} n$ of our algorithm for discrete midpoint convex functions.

%%%%%%%%%%%%%% SSSSS %%%%%%%%%%%%%%%%%%%%%
\section{Quadratic Functions}
\label{SCquadfnDIC}

%%\marginpar{Major revision in June 2018}

In this section we study 
discrete midpoint convexity of a quadratic function
$f(x) = x\sp{\top} Q x$,
where $Q$ is an $n \times n$ symmetric matrix
and $x \in \mathbb{Z}\sp{n}$.
We start with a general characterization of 
(locally and globally) discrete midpoint convex quadratic functions.

\begin{lemma}\label{LMzQz1Q1}
For $f(x) = x\sp{\top} Q x$,
the discrete midpoint convexity {\rm (\ref{discmptconvfn2})}
for $x, y \in \mathbb{Z}\sp{n}$
is equivalent to
\begin{equation}\label{midptQuad}
z\sp{\top} Q z \geq \bm{1}_J\sp{\top} Q \bm{1}_J
\end{equation}
for $z = x -y$,
where $J = \{i \mid z_i \mbox{\rm \ is odd }  \}$.
Hence, $f$ is locally (resp., globally) discrete midpoint convex
if and only if {\rm (\ref{midptQuad})} holds for all
$z \in \mathbb{Z}\sp{n}$ with
$\| z \|_{\infty} = 2$ (resp., $\| z \|_{\infty} \geq 2$).
\end{lemma}
\begin{proof}
With the use of identities
\begin{align*}
& \frac{1}{2} \left( \left\lfloor \frac{x+y}{2} \right\rfloor
 + \left\lceil \frac{x+y}{2} \right\rceil \right)
= \frac{x+y}{2} \ ,
\\ &
 f(x) + f(y) - 2 f \left( \frac {x+y}{2} \right)
= \frac{1}{2}(x-y)\sp{\top} Q (x-y),
\end{align*}
we can rewrite
the discrete midpoint convexity (\ref{discmptconvfn2}) to
\[
(x-y)\sp{\top} Q (x-y) \geq \left( \left\lceil \frac{x+y}{2} \right\rceil - \left\lfloor \frac{x+y}{2} \right\rfloor \right)\sp{\top} Q
\left( \left\lceil \frac{x+y}{2} \right\rceil - \left\lfloor \frac{x+y}{2} \right\rfloor \right).
\]
The substitution of $x -y = z$ and
$
\left\lceil \frac{x+y}{2} \right\rceil
- \left\lfloor \frac{x+y}{2} \right\rfloor
= \textbf{1}_{J}
$
yields (\ref{midptQuad}).
\end{proof}

It follows from Lemma \ref{LMzQz1Q1} and Proposition \ref{PRmdptdistant} (3)
that the set of matrices $Q$ for which 
 $f(x) = x\sp{\top} Q x$ 
is (locally or globally) discrete midpoint convex is a polyhedral convex cone.

We next consider the relationship between 
the discrete midpoint convexity of $f(x) = x\sp{\top} Q x$ 
and the diagonal dominance%
\footnote{%%%%%%%%%%%%%%%%%%%%
To be precise, the condition (\ref{midptdiagdomdef}) says that 
$Q$ is a diagonally dominant matrix with nonnegative diagonals.
In this paper, however, we refer to (\ref{midptdiagdomdef}) simply as diagonal dominance.
}  %%%footnote%%%% 
of $Q$: 
\begin{equation}\label{midptdiagdomdef}
q_{ii} \geq \sum_{j \not= i} |q_{ij}|
\qquad (i=1,\ldots,n).
\end{equation}

The following facts are known for integral convexity and ${\rm L}\sp{\natural}$-convexity.

\begin{proposition}[{\cite[Proposition 4.5, Remark 4.3]{FT90}}] \label{PRdiagdomIC}
If  $Q$ is diagonally dominant
{\rm (\ref{midptdiagdomdef})},
then $f(x) = x\sp{\top} Q x$ is integrally convex.
The converse is also true if $n \leq 2$.
\finbox
\end{proposition}

\begin{proposition}[{\cite[Section 7.3]{Mdcasiam}}] \label{PRdiagdomLnat}
$f(x) = x\sp{\top} Q x$ is ${\rm L}\sp{\natural}$-convex
if and only if
$Q$ is diagonally dominant
{\rm (\ref{midptdiagdomdef})},
and
$q_{ij} \leq 0$ for all $i \not= j$.
\finbox
\end{proposition}

In contrast, discrete midpoint convexity is not governed by diagonal dominance.
This is demonstrated in the following examples.

\begin{example} \rm \label{EXdmcNOTdiagdom}
Let $f(x) = x\sp{\top} Q x$ with 
$Q = {\small\footnotesize
\left[ \begin{array}{rrr}
1 & -1 & 1  \\
-1 & 2 & -1  \\
1 & -1 & 2  \\
\end{array}\right]}$
for $x  \in \ZZ\sp{3}$.
The matrix $Q$ is not diagonally dominant.
Nevertheless, the function $f(x)$ is globally (and hence locally) midpoint convex.
Indeed, it is possible 
to verify the inequality
(\ref{midptQuad}) in Lemma \ref{LMzQz1Q1}
for all $z \in \ZZ\sp{3}$ with $\| z \|_{\infty} \geq 2$.
\finbox
\end{example}

\begin{example} \rm \label{EXdiagdomNOTdmc}
Let $f(x) = x\sp{\top} Q x$ with 
$Q = {\small\footnotesize
\left[ \begin{array}{rrr}
1 & 1 & 0  \\
1 & 1 & 0  \\
0 & 0 & 0 \\
\end{array}\right]}$
for $x  \in \ZZ\sp{3}$.
The matrix $Q$ is diagonally dominant.
Nevertheless, the function $f(x)= (x_1 + x_2)^2$ is not locally midpoint convex
in $(x_1, x_2, x_3)$.
Indeed, 
for $x = (1,0,0)$ and $y = (0,1,2)$, we have
$z=\left\lceil \frac{x+y}{2} \right\rceil = (1,1,1)$,
$w = \left\lfloor \frac{x+y}{2} \right\rfloor = (0,0,1)$,
and
$f(x) + f(y) = 1 + 1  < f(z) + f(w) = 4 + 0$.
It is also noted that the function $(x_1 + x_2)^2$ is locally midpoint convex in $(x_1, x_2)$.
\finbox
\end{example}

We can still make the following statement.

\begin{proposition}\label{PRDiagdomLDCM}
If $f$ is locally discrete midpoint convex and 
$q_{ij} \leq  0$ for all $i \not= j$, 
then $Q$ is diagonally dominant.
\end{proposition}
\begin{proof}
Since $f$ is locally discrete midpoint convex,
it is integrally convex by Theorem \ref{THdirintclass}.
Furthermore it is submodular by $q_{ij} \leq 0$ ($i \not= j$), and hence
${\rm L}\sp{\natural}$-convex by Theorem \ref{THlnatcond}.
Finally 
Proposition \ref{PRdiagdomLnat}
shows the diagonal dominance of $Q$.
\end{proof}

For two-dimensional quadratic functions, 
local and global discrete midpoint convexities
can be characterized as follows.

\begin{proposition}\label{LMdiagDomQuad2dimL}
For $n =2$, $f(x) = x\sp{\top} Q x$ is locally discrete midpoint convex
if and only if
$ q_{11} \geq |q_{12}|$ and $ q_{22} \geq |q_{12}|$.
\end{proposition}
\begin{proof}
Suppose that $f$ is locally discrete midpoint convex.
By inequality (\ref{midptQuad}) for $z =(2,1)$ and $z =(2,-1)$ 
we obtain
$q_{11}  + q_{12} \geq  0$ and $q_{11}  - q_{12} \geq  0$, i.e., 
$q_{11}  - |q_{12}| \geq  0$. 
By symmetry, we also have
$q_{22}  - |q_{21}| \geq  0$.
Conversely, these conditions imply (\ref{midptQuad}) for all $z$ with $\| z \|_{\infty} = 2$.
\end{proof}

\begin{proposition}\label{LMdiagDomQuad2dimG}
For $n =2$, $f(x) = x\sp{\top} Q x$
is globally discrete midpoint convex
if and only if
$ q_{11} \geq |q_{12}|$, $ q_{22} \geq |q_{12}|$,
and
$ q_{11} + q_{22}  \geq (5/2)q_{12}$.
\end{proposition}

\begin{proof}
Suppose that $f$ is globally discrete midpoint convex.
By inequality (\ref{midptQuad}) for $z =(3,-3)$
we obtain
$2(q_{11}  + q_{22}) \geq  5 q_{12}$.
Using also 
Proposition~\ref{LMdiagDomQuad2dimL}
we obtain three inequalities
$ q_{11} \geq |q_{12}|$, $ q_{22} \geq |q_{12}|$,
and
$ q_{11} + q_{22}  \geq (5/2)q_{12}$.
Conversely, these conditions imply (\ref{midptQuad}) 
for all $z =(u,v)$ with $\max( |u|, |v| ) \geq 2$, as follows.

Inequality (\ref{midptQuad}) for $z =(u,v)$ is easy to see
if $u$ or $v$ is even.
For example,
if $u$ odd and $v$ even,
(\ref{midptQuad}) is equivalent to
$q_{11} u^{2}  + q_{22} v^{2} + 2 q_{12} uv \geq  q_{11}$,
which can be shown as follows:
\begin{align*}
q_{11} u^{2}  + q_{22} v^{2} + 2 q_{12} uv
&
\geq
q_{11} u^{2}  + q_{22} v^{2} - 2 |q_{12}| |uv|
\\ &
= |q_{12}| (|u|-|v|)^{2}
 + (q_{11} -|q_{12}| )u^{2}
 + (q_{22} -|q_{12}| )v^{2}
\\
&\geq |q_{12}| +  (q_{11} -|q_{12}| ) = q_{11},
\end{align*}
where the second inequality follows from
$|u|-|v| \not= 0$ and $u \not= 0$.

Assume that $u$ and $v$ are odd.  Then
\begin{align*}
\mbox{ (\ref{midptQuad})}
& \Leftrightarrow
q_{11} u^{2}  + q_{22} v^{2} + 2 q_{12} uv \geq  q_{11} + q_{22} + 2 q_{12}
\\ & \Leftrightarrow
F := q_{11} (u^{2}-1)  + q_{22} (v^{2}-1) + 2 q_{12} (uv-1) \geq 0 .
\end{align*}
We consider the following five cases separately:
\\
\qquad
\begin{tabular}{ll}
Case 1:  $|u| \geq 3$ and $|v| = 1$,
\\
Case 2:  $|v| \geq 3$ and $|u| = 1$,
\\
Case 3: $|u|\geq 3$, $|v| \geq 3$, and $q_{12} uv \geq 0$,
\\
Case 4: $|u|\geq 3$, $|v| \geq 3$,  $q_{12} \geq 0$, and $uv <0$.
\\
Case 5: $|u|\geq 3$, $|v| \geq 3$,  $q_{12} \leq 0$, and $uv > 0$.
\end{tabular}
\medskip

\noindent
Case 1 ($|u| \geq 3$ and $|v| = 1$): We have
\[
F
 \geq
q_{11} (u^{2}-1) - 2 |q_{12}| (|u|+1)
 \geq
q_{11} (u^{2}-1) - 2 q_{11} (|u|+1)
=
q_{11} ((|u|-1)^{2} - 4)  \geq 0.
\]
\noindent
Case 2 ($|v| \geq 3$ and $|u| = 1$): Similarly to Case 1.

\noindent
Case 3 ($|u|\geq 3$, $|v| \geq 3$, $q_{12} uv \geq 0$):
We have $q_{12} (uv-1) \geq 0$ in this case, and hence
\begin{align*}
F =
q_{11} (u^{2}-1)  + q_{22} (v^{2}-1) + 2 q_{12} (uv-1)
\geq
q_{11} (u^{2}-1)  + q_{22} (v^{2}-1) \geq 0.
\end{align*}

\noindent
Case 4 ($|u|\geq 3$, $|v| \geq 3$,  $q_{12} \geq 0$, $uv <0$):
We have
\begin{align*}
F  &=
q_{11} (u^{2}-1)  + q_{22} (v^{2}-1) - 2 q_{12} (|uv|+1)
\\
&= q_{12} (|u|-|v|)^{2}
 + (q_{11} -q_{12} )u^{2}
 + (q_{22} -q_{12} )v^{2}
  -  q_{11} - q_{22} - 2 q_{12}
\\
&\geq
  (q_{11} -q_{12} ) \times 9
 +  (q_{22} -q_{12} ) \times 9
  -  q_{11} - q_{22} - 2 q_{12}
\\ &=
 8(q_{11}  + q_{22}) - 20 q_{12} \geq 0.
\end{align*}
Case 5 ($|u|\geq 3$, $|v| \geq 3$,  $q_{12} \leq 0$, $uv > 0$):
We have
\begin{align*}
F
&=
q_{11} (u^{2}-1)  + q_{22} (v^{2}-1) - 2 |q_{12}| (uv-1)
\\
&=
 |q_{12}| (u-v)^{2}
 + (q_{11} -|q_{12}| )u^{2}
 + (q_{22} -|q_{12}| )v^{2}
  -  q_{11} - q_{22} + 2 |q_{12}|
\\
&\geq
  (q_{11} -|q_{12}| ) \times 9
 +  (q_{22} -|q_{12}| ) \times 9
  -  q_{11} - q_{22} + 2 |q_{12}|
\\ &=
 8 (q_{11} -|q_{12}| ) + 8 (q_{22} -|q_{12}| ) \geq 0.
\end{align*}
\vspace{-1.5\baselineskip}\\
\end{proof}

For the general $n$-dimensional case, we have the following sufficient condition 
for global discrete midpoint convexity of quadratic functions.
The minimum and maximum eigenvalues of $Q$ are denoted, respectively, by
$\lambda^{Q}_{\min}$ and $\lambda^Q_{\max}$.

\begin{proposition}\label{PReigenratio}
$f(x) = x\sp{\top} Q x$
is globally discrete midpoint convex if
\begin{equation}\label{eqQuadWeakDirIntConv3J}
 \lambda^Q_{\min} \geq \frac{n-1}{n+3} \  \lambda^{Q}_{\max}.
\end{equation}
\end{proposition}

\begin{proof}
We prove that (\ref{eqQuadWeakDirIntConv3J})
implies (\ref{midptQuad}).
We may assume $n \geq 2$.

For an $n \times n$ symmetric matrix $C$,
let $\lambda^C_{\min}$ and $\lambda^C_{\max}$ denote its minimum and maximum eigenvalues,
respectively. Furthermore,
for a set  $I \subseteq N = \{ 1,\ldots, n  \}$ of indices,
let $C_I$ denote the submatrix of $C$ with row and column indices in $I$.
Note that  (\ref{eqQuadWeakDirIntConv3J})
implies
\begin{equation}\label{eqQuadWeakDirIntConv3I}
\lambda^{Q_{I}}_{\min} \geq \frac{n-1}{n+3} \  \lambda^{Q_{I}}_{\max}
\end{equation}
for any
$I \subseteq N$,
since
$0 \leq \lambda^{Q}_{\min} \leq
\lambda^{Q_{I}}_{\min} \leq \lambda^{Q_{I}}_{\max} \leq \lambda^{Q}_{\max}$.

Take $z \in \mathbb{Z}\sp{n}$ with $\| z \|_{\infty} \geq 2$,
and let
$I = \{i \mid z_i \neq 0 \}$ and $m = |I|$,
whereas $J = \{i \mid z_i \mbox{ is odd }  \}$.
If $\| z \|_{\infty} = 2$, we have
$z\sp{\top} z \geq 2\sp{2} + (|I| -1) = m +3$
and $\bm{1}_{J}\sp{\top} \bm{1}_{J} = |J| \leq m-1$.
If $\| z \|_{\infty} \geq 3$, we have
$z\sp{\top} z \geq 3\sp{2} + (|I| -1) = m +8$
and $\bm{1}_{J}\sp{\top} \bm{1}_{J} = |J| \leq m$.
That is, by defining
\[
(\zeta,\eta)=
\left\{ \begin{array}{ll}  %% cases
(m+3,m-1)
 & (\mbox{if $\| z \|_{\infty} = 2$}), \\
 (m+8,m) & (\mbox{if $\| z \|_{\infty} \geq 3$}) ,
 \end{array}\right.
\]
we have
$z\sp{\top} z \geq \zeta$ and
$\bm{1}_{J}\sp{\top} \bm{1}_{J} \leq \eta$.
Then
(\ref{midptQuad})
can be shown as follows:
\begin{align*}
 z\sp{\top}  Q z
& \geq
 \zeta \  \frac{z\sp{\top}  Q z}{z\sp{\top}  z}
\geq \zeta \  \lambda^{Q_{I}}_{\min}
\geq \zeta \  \frac{n-1}{n+3} \  \lambda^{Q_{I}}_{\max}
\\ &
\geq \zeta \  \frac{n-1}{n+3} \
 \frac{ \bm{1}_{J}\sp{\top} Q \bm{1}_{J} }{ \bm{1}_{J}\sp{\top}  \bm{1}_{J} }
\geq \frac{\zeta }{\eta} \  \frac{n-1}{n+3} \
 \bm{1}_{J}\sp{\top} Q \bm{1}_{J}
\geq  \bm{1}_{J}\sp{\top} Q \bm{1}_{J} ,
\end{align*}
where the six inequalities follow, respectively, from:
$z\sp{\top} z \geq \zeta >0$,
the definition of $I$,
(\ref{eqQuadWeakDirIntConv3I}),
$I \supseteq J$,
$\bm{1}_{J}\sp{\top}  \bm{1}_{J} \leq \eta$,
and $\zeta(n-1) \geq \eta(n+3)$.
(Proof of $\zeta(n-1) \geq \eta(n+3)$:
if $\| z \|_{\infty} = 2$, we have
$\zeta(n-1) - \eta(n+3) = (m+3)(n-1) - (m-1)(n+3) =
4(n - m) \geq 0$; and
 if $\| z \|_{\infty} \geq 3$, we have
$\zeta(n-1) - \eta(n+3) = (m+8)(n-1) - m(n+3) = 4(2n - m -2)
\geq 4(n-2) \geq 0$.)
\end{proof}

The converse of Proposition~\ref{PReigenratio} is not true, which is shown by the following example.

\begin{example} \rm \label{EXdicdim3eigval}
The matrix $Q$ in Example \ref{EXdmcNOTdiagdom},
yielding a globally midpoint convex function,
does not satisfy the condition
(\ref{eqQuadWeakDirIntConv3J}) in Proposition~\ref{PReigenratio}.
The three eigenvalues of $Q$ are given by 
$2 - \sqrt{3} = \lambda^{Q}_{\min}$, $1$, and  
$2 + \sqrt{3} = \lambda^{Q}_{\max}$,
for which the left-hand side of (\ref{eqQuadWeakDirIntConv3J}) is equal to 
$2 - \sqrt{3} = 0.268 \cdots$
and the right-hand side to $(2 + \sqrt{3})/3 = 1.244  \cdots$.
\finbox
\end{example}

We note the following
as a corollary of  Proposition \ref{PReigenratio}.

\begin{proposition}\label{PRquadRowSumJ}
Assume $n \geq 2$.
$f(x) = x\sp{\top} Q x$ is globally discrete midpoint convex
if $Q$ is represented as $Q = \alpha (I + R)$
with $\alpha \geq 0$ and a positive semidefinite $R$ satisfying
\begin{equation}\label{eqquadRowSumJ}
\max_{1 \leq i \leq n} \sum_{j=1}\sp{n} |r_{ij}| \leq \frac{4}{n-1} .
\end{equation}
\end{proposition}
\begin{proof}
We have
$\lambda^{Q}_{\max} = \alpha (1 + \lambda^{R}_{\max})$ and
$\lambda^{R}_{\max} \leq
\max_{1 \leq i \leq n} \sum_{j=1}\sp{n} |r_{ij}|$,
where the latter is a consequence of the Gershgorin theorem.
Hence (\ref{eqquadRowSumJ}) implies
$\lambda^{Q}_{\max} \leq \alpha (n+3)/(n-1)$.
On the other hand,
we have $\lambda^{Q}_{\min} \geq \alpha$
since $R$ is positive semidefinite.
Hence follows (\ref{eqQuadWeakDirIntConv3J}).
\end{proof}

%%%%%%%%%%%%%% SSSSS %%%%%%%%%%%%%%%%%%%%%
\section{Proofs about Discrete Midpoint Convex Sets}
\label{SCproofgen}

\subsection{Proof of Proposition \ref{PRdirintcnvSetIC} (2)}
\label{SCdicproofPRdirintcnvSetIC}

Let $S \subseteq  \ZZ\sp{n}$ be a discrete midpoint convex set.
We prove that $S$ is an integrally convex set.
That is, we are to show that 
$z \in \overline{S}$ implies $z \in \overline{S \cap N(z)}$. 
 Fix any $z \in \overline{S}$,
which can be represented as a convex combination
\begin{equation}\label{sdic-eq3}
 z =\sum_{i=1}\sp{r} \lambda_{i} y^i,
\quad 
\sum_{i=1}\sp{r} \lambda_{i} = 1,
\quad  \lambda_{i} > 0  \ \  (i=1,2,\ldots,r)
\end{equation}
of some $y\sp{1},y\sp{2},\ldots,y\sp{r} \in S$
with coefficients 
$\lambda_1, \lambda_2, \ldots, \lambda_r \in \RR$.

With reference to the $n$-th component, we repeat 
modifying the generators 
$\{y\sp{1},y\sp{2},\ldots,y\sp{r}\}$ so that an additional condition
\begin{equation}\label{sdic-eq4}
 |y^i_{n} - z_{n}| < 1 
\quad (i =1,2,\ldots,r) 
\end{equation}
is satisfied.
By applying the same procedure for all other components,
we arrive at a representation of $z$ as a convex combination 
of points in $S \cap N(z)$.

The modification procedure with reference to the $n$-th component
is now described.
Without loss of generality we may assume
\begin{equation}\label{sdic-eq5}
 |y\sp{1}_{n} - z_{n}| \geq |y\sp{2}_{n} - z_{n}| \geq \cdots \geq |y\sp{r}_{n} - z_{n}| .
\end{equation}
Suppose that
$|y\sp{1}_{n} - z_{n}| \geq 1$;
otherwise we are done.
We may assume $y\sp{1}_{n} - z_{n} \geq 1$,
since the other case $y\sp{1}_{n} - z_{n} \leq -1$ can be treated in a similar manner.

We claim that there exists
$\{y^{j_{1}},\ldots,y^{j_k}\} \subseteq \{y\sp{2},y^3,\ldots,y\sp{r}\}$
such that, for $i=1,\ldots, k$,
the $n$-th component of $y^{j_{i}}$ is less than that of $z$
(i.e., $y^{j_{i}}_{n} < z_{n}$)
and 
\begin{equation}\label{sdic-eq6}
 \sum_{i=1}^{k-1} \lambda_{j_{i}} < \lambda_{1} \leq \sum_{i=1}^k \lambda_{j_{i}} .
\end{equation}
This claim can be proved as follows.
Let $I = \{ i \mid y^{i}_{n} < z_{n} \}$.
By (\ref{sdic-eq3}), (\ref{sdic-eq5}), and $y\sp{1}_{n} - z_{n} \geq 1$ we obtain
\begin{align*}
 0 &= \sum_{i=1}^{r} \lambda_{i}(y^{i}_{n} - z_{n}) 
  \ = \sum_{i: \,  y^{i}_{n} > z_{n} } \lambda_{i}(y^{i}_{n} - z_{n}) 
        + \sum_{i: \,  y^{i}_{n} < z_{n}} \lambda_{i}(y^{i}_{n} - z_{n}) \\
   &\geq  \lambda_{1}(y^{1}_{n} - z_{n}) + \sum_{i\in I} \lambda_{i}(y^{i}_{n} - z_{n}) 
   \ = \lambda_{1}|y^{1}_{n} - z_{n}| - \sum_{i\in I} \lambda_{i}|y^{i}_{n} - z_{n}| \\ 
   & \geq \  \lambda_{1}|y^{1}_{n} - z_{n}| - \sum_{i\in I} \lambda_{i}|y^{1}_{n} - z_{n}| 
   \ = |y^{1}_{n} - z_{n}| \  \big(\lambda_{1} - \sum_{i\in I} \lambda_{i}\big),
\end{align*}
which shows
$\lambda_{1} \leq \sum_{i\in I} \lambda_{i}$.
Then there exists $k$ that satisfies (\ref{sdic-eq6}).

For $i=1,\ldots,k$ we have
$\|y\sp{1} - y^{j_{i}}\|_\infty \geq 2$,
since 
$y\sp{1}_{n} \geq  z_{n} + 1$
and 
$y^{j_{i}}_{n} < z_{n}$.
Define
\begin{equation}\label{sdic-eq7} 
  y^{+j_{i}} = \left\lceil \frac{y\sp{1}+y^{j_{i}}}{2} \right\rceil,\quad
  y^{-j_{i}} = \left\lfloor \frac{y\sp{1}+y^{j_{i}}}{2} \right\rfloor
\end{equation}
for $i=1,\ldots,k$. 
We have $y^{+j_{i}}, y^{-j_{i}} \in S$ for $i=1,\ldots,k$
by the definition (\ref{dirintcnvsetdef}) of a discrete midpoint convex set.
Among the terms in the convex combination (\ref{sdic-eq3})
we consider the terms for 
$y\sp{1}, y^{j_{1}},\ldots,y^{j_k}$.
Since $y\sp{1}+y^{j_{i}} = y^{+j_{i}} + y^{-j_{i}}$ 
for $i=1,\ldots,k-1$ and
$y\sp{1} =  y^{+j_k} + y^{-j_k} - y^{j_k}$,
we have 
\begin{align*}
& \lambda_{1} y\sp{1} + \sum_{i=1}^k \lambda_{j_{i}} y^{j_{i}}
\\
 &= \sum_{i=1}^{k-1} \lambda_{j_{i}} (y^{1}+y^{j_{i}})
 + \left(\lambda_{1}-\sum_{i=1}^{k-1} \lambda_{j_{i}} \right) y^{1}
 +  \lambda_{j_k} y^{j_k}
\\
& = \sum_{i=1}^{k-1} \lambda_{j_{i}} (y^{+j_{i}}+y^{-j_{i}})
 + \left(\lambda_{1}-\sum_{i=1}^{k-1} \lambda_{j_{i}} \right) (y^{+j_k}+y^{-j_k})
 + \left(\sum_{i=1}^{k} \lambda_{j_{i}} - \lambda_{1}\right)y^{j_k}.
\end{align*}
By (\ref{sdic-eq6}) all the coefficients 
in the last expression are nonnegative
and their sum is equal to the sum of the coefficients  
of the first expression, which is  
$\lambda_{1}+ \sum_{i=1}^k \lambda_{j_{i}}$.

We change the generators 
$\{y\sp{1},y\sp{2},\ldots,y\sp{r}\}$
in the representation (\ref{sdic-eq3})
by deleting
$y\sp{1}$ and $\{y^{j_{1}},\ldots,y^{j_{k-1}}\}$
(and also $y^{j_k}$ if $\sum_{i=1}^{k} \lambda_{j_{i}} - \lambda_{1}=0$)
and by adding 
$\{y^{+j_{1}},\ldots,y^{+j_k}\}$
and $\{y^{-j_{1}},\ldots,y^{-j_k}\}$.
Since $|y^{+j_{i}}_{n}-z_{n}| < |y\sp{1}_{n}-z_{n}|$ 
and 
$|y^{-j_{i}}_{n}-z_{n}| < |y\sp{1}_{n}-z_{n}|$
for $i=1,\ldots,k$,
the resulting set of generators is ``lexicographically smaller''
with respect to the discrepancy in the $n$-th components 
arranged as in (\ref{sdic-eq5}).
By finiteness, the above procedure eventually terminates 
with $|y\sp{1}_{n}-z_{n}| < 1$, which implies (\ref{sdic-eq4}).

We next apply the above procedure with reference to the $(n-1)$-st component.
What is crucial here is that 
the condition (\ref{sdic-eq4})
is maintained in the modification of the generators.
Indeed, 
for each $i$, the inequalities
$ |y^{+j_i}_n - z_n| < 1$ and
$ |y^{-j_i}_n - z_n| < 1$ 
follow from (\ref{sdic-eq4})
and (\ref{sdic-eq7}).
Therefore, we can obtain 
a representation of the form of 
(\ref{sdic-eq3}) with 
\begin{equation}\label{sdic-eq45} 
 |y^i_n - z_n| < 1, \quad 
 |y^i_{n-1} - z_{n-1}| < 1 
\qquad (i =1,2,\ldots,r) .
\end{equation}
By continuing in this way for other components,
we finally obtain a representation of the form of (\ref{sdic-eq3})
with  $ |y^i_{j} - z_{j}| < 1$
for $j =1,2,\ldots,n$ and $i =1,2,\ldots,r$.
This completes the proof of Proposition \ref{PRdirintcnvSetIC} (2).

\subsection{Proof of Theorem~\ref{THdirintcnvSetdec}}
\label{SCdicproofTHdirintcnvSetdec}

Let $S  \subseteq \mathbb{Z}\sp{n}$
be a discrete midpoint convex set and $x, y \in S$,
and consider the decomposition 
$y - x = \sum_{k=1}^{m} (\bm{1}_{A_{k}} - \bm{1}_{B_{k}})$
in (\ref{DICdecAkBksum}).
For the proof of Theorem \ref{THdirintcnvSetdec} 
it suffices to prove that
\begin{equation}\label{xJ1A1BinS}
x + \sum_{k \in J}  (\bm{1}_{A_{k}} - \bm{1}_{B_{k}}) \in S
\end{equation}
for any subset $J$ of $\{ 1,2,\ldots, m \}$.
Indeed, Theorem \ref{THdirintcnvSetdec} follows from 
(\ref{xJ1A1BinS}) for $J$ and  
$\{ 1,2,\ldots, m \} \setminus J$.

Our proof strategy is to relate the decomposition (\ref{DICdecAkBksum})
to discrete midpoint convexity 
by showing an alternative construction of 
the decomposition (\ref{DICdecAkBksum}) 
using operations of rounding-up  
$\lceil x \rceil$ and rounding-down $\lfloor x\rfloor$.
For any $x \in \mathbb{Z}\sp{n}$ we denote the positive support of $x$  
by $\suppp(x) = \{ i \mid x_{i} > 0 \}$ 
and the negative support of $x$ by $\suppm(x) = \{ i \mid x_{i} < 0 \}$. 
We are concerned with a decomposition of
an integer vector $x$ into a family $D(x)$ of vectors such that:
\begin{description}
\item[\rm (C1)] $d \in \{-1, 0, +1 \}\sp{n} \setminus \{ \veczero \}
 \qquad (\forall d \in D(x))$.

\item[\rm (C2)] $\sum\{ d \mid d \in D(x) \} = x$.

\item[\rm (C3)] $\suppp(d) \subseteq \suppp(x),\quad 
 \suppm(d) \subseteq \suppm(x) \qquad (\forall d \in D(x))$.

\item[\rm (C4)]  $\{ \suppp(d) \mid d \in D(x) \}$
and 
$\{ \suppm(d) \mid d \in D(x) \}$ 
each form a chain (nested family) with respect to set inclusion.

\item[\rm (C5)] $| D(x) | = \| x \|_{\infty}$.

\item[\rm (C6)]
The vectors in $D(x)$ form a chain 
(linearly ordered set) with respect to vector ordering.
\end{description}

We first introduce a recursive scheme 
for a decomposition of an integer vector $x$ into a family $D_{0}(x)$ of vectors,
which satisfies (C1) to (C4) above.
The family $D_{0}(x)$ is modified
to $D_{1}(x)$ to meet (C5),
and then to $D_{2}(x)$, which satisfies (C1) to (C6).
This decomposition scheme, when applied to $y-x$ for $x,y \in \ZZ\sp{n}$, yields
the decomposition (\ref{DICdecAkBksum}), that is, 
\begin{equation}\label{D2=1A1B}
 D_{2}(y-x) = \{ \bm{1}_{A_{k}} - \bm{1}_{B_{k}} \mid k=1,\ldots,m \} .
\end{equation}
Finally we show that, if $x, y \in S$, where $S$ is a discrete midpoint convex set,
then the property (\ref{xJ1A1BinS}) follows from the construction of $D_{2}(y-x)$.
Having explained our proof strategy, we now begin the technical arguments.
The proof consists of five steps.

\paragraph{Step 1:}
For any $x \in \ZZ\sp{n}$, 
we define a family (multiset) $D_{0}(x)$
of vectors  by the following recursive formula:
\begin{equation} \label{D0def}
 D_{0}(x) 
  = \left\{ \begin{array}{ll}  %% cases
  \emptyset & (x = \veczero) , \\
  \{x\} & (\|x\|_{\infty} = 1) , \\
   \{ \  \lceil x/2 \rceil, \  \lfloor x/2 \rfloor  \, \} 
    & (\|x\|_{\infty} = 2) , \\
  D_{0}(\lceil x/2 \rceil)\  \cup \  D_{0}(\lfloor x/2 \rfloor) 
    & (\|x\|_{\infty} \geq 3) .
 \end{array}\right.
\end{equation}

\begin{lemma}\label{LMtamProp1}
For any $x \in \ZZ\sp{n}$, $D_{0}(x)$ satisfies
{\rm (C1)},
{\rm (C2)},
{\rm (C3)}, and
{\rm (C4)}.
\end{lemma}
\begin{proof}
(C1) and (C3) are obvious from the definition (\ref{D0def}).

(C2)
If
$\|x \|_{\infty} = 1$,
we have $D_{0}(x) = \{x\}$ 
and the claim is obviously true.
If
$\|x \|_{\infty} = 2$,
we have
$D_{0}(x) = \left\{\left\lceil x/2  \right\rceil, \left\lfloor x/2  \right\rfloor \right\}$,
and the claim is true since
$\left\lceil x/2  \right\rceil + \left\lfloor x/2  \right\rfloor = x$.
We prove the claim for 
$\| x \|_{\infty} \geq 3$
by induction on $\| x \|_{\infty}$.
If $\| x \|_{\infty} \geq 3$,
we have
$\left\| \, \left\lceil x/2  \right\rceil \, \right\|_{\infty} < \| x \|_{\infty}$
and 
$\left\| \, \left\lfloor x/2  \right\rfloor \, \right\|_{\infty} < \| x \|_{\infty}$.
Hence, the induction hypothesis implies
\[
\sum\left\{ d \mid d \in D_{0}\left(\left\lceil x/2  \right\rceil\right) \right\} 
= \left\lceil x/2  \right\rceil, 
\quad
\sum\left\{ d \mid d \in D_{0}\left(\left\lfloor x/2  \right\rfloor\right) 
\right\} = \left\lfloor x/2  \right\rfloor.
\]
Therefore, 
$\sum\left\{ d \mid d \in D_{0}(x) \right\} 
= \sum\left\{ d \mid d \in D_{0}\left(\left\lceil x/2  \right\rceil\right) \right\} 
+ \sum\left\{ d \mid d \in D_{0}\left(\left\lfloor x/2  \right\rfloor\right) 
\right\} 
= \left\lceil x/2  \right\rceil + \left\lfloor x/2  \right\rfloor
= x$.

(C4)
The definition (\ref{D0def}) implies:
\begin{align*}
x_{i} \geq x_{j} \Rightarrow d_{i} \geq d_{j} \qquad &
 (\forall i,j \in \{1,2,\ldots,n\},\; \forall d \in D_{0}(x)) ,
\\
x_{i} = x_{j} \Rightarrow d_{i} = d_{j} \qquad &
 (\forall i,j \in \{1,2,\ldots,n\},\; \forall d \in D_{0}(x)).
\end{align*}
Therefore, for any $d,d' \in D_{0}(x)$,
we have
$\suppp(d) \subseteq \suppp(d')$ 
or 
$\suppp(d) \supseteq \suppp(d')$,
which means that 
$\{ \suppp(d) \mid d \in D_{0}(x) \}$ forms a chain.
Similarly for 
$\{ \suppm(d) \mid d \in D_{0}(x) \}$.
\end{proof}

\begin{example} \rm \label{EXtamExample1}
For $x = (5,3,-3,-5)$, $D_{0}(x)$ is given as follows:
\begin{align*}
 &D_{0}((5,3,-3,-5))
\\ 
 &= D_{0}((3,2,-1,-2)) \cup D_{0}((2,1,-2,-3)) 
\\
 &= D_{0}((2,1,0,-1)) \cup D_{0}((1,1,-1,-1)) 
 \ \cup \   D_{0}((1,1,-1,-1)) \cup D_{0}((1,0,-1,-2)) 
\\
 &= D_{0}((1,1,0,0)) \cup D_{0}((1,0,0,-1)) \ \cup \  D_{0}((1,1,-1,-1)) 
   \ \cup \  D_{0}((1,1,-1,-1)) 
\\
  &\quad \ \cup \  D_{0}((1,0,0,-1)) \cup D_{0}((0,0,-1,-1)) 
\\
 &= \{ \ (1,1,0,0), (1,0,0,-1), (1,1,-1,-1), (1,1,-1,-1), (1,0,0,-1), (0,0,-1,-1) \ \}.
\end{align*}

%%%  FIGURE %%%%%%%%%%%%%%%%%%
%%\input{MMTTfgDtree}
\begin{figure}\begin{center}
\includegraphics[height=50mm]{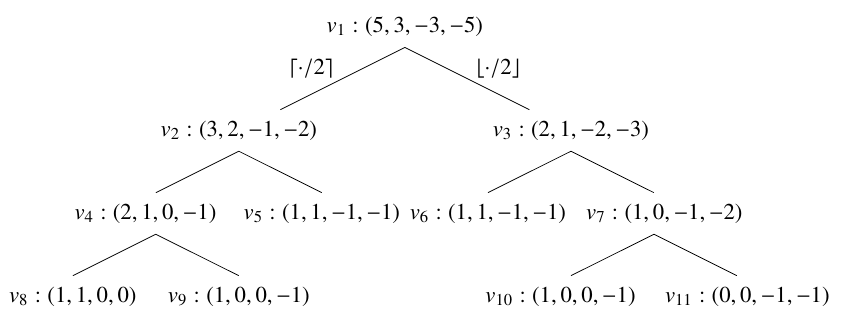}
\caption{Tree of $D_{0}(x)$ for $x = (5,3,-3,-5)$ in Example \ref{EXtamExample1}.}
\label{FGdtree}
\end{center}\end{figure}
%%%  FIGURE %%%%%%%%%%%%%%%%%%

The recursive process of the definition of $D_{0}(x)$ 
can be represented by the binary tree shown in Fig.~\ref{FGdtree}.
This tree has eleven vertices 
$\{ v_{1}, \ldots, v_{11} \}$, and each vertex is associated with 
a vector that appears in the recursive definition of $D_{0}(x)$ above.

Concerning (C4) we have
$\{ \suppp(d) \mid d \in D_{0}(x) \} = \{\{1,2\},\{1\}, \emptyset \}$
and  
$\{ \suppm(d) \mid d \in D_{0}(x) \} = \{\{3,4\},\{4\}, \emptyset \}$,
which are chains indeed.
$D_{0}(x)$ does not satisfy (C5),
since $| D_{0}(x) | =6$ and $\| x \|_{\infty} = 5$.
$D_{0}(x)$ does not satisfy (C6) either,
since
neither 
$(1,0,0,-1) \leq (1,1,-1,-1)$
nor
$(1,0,0,-1) \geq (1,1,-1,-1)$
is true.
\finbox
\end{example}

Just as in the above example, 
we can associate a binary tree with the recursive definition of $D_{0}(x)$.
Let us denote this tree by $T_{0}(x)$
and the vector associated with a vertex $v$ of this tree
by $\varphi(v)$.
Then 
$D_{0}(x)$ is given as 
$D_{0}(x) = \{ \varphi(v) \mid \mbox{$v$ is a leaf of $T_{0}(x)$} \}$.
We assume that the left child of a vertex $v$ corresponds to
$\left\lceil \varphi(v)/2  \right\rceil$
and the right child to 
$\left\lfloor \varphi(v)/2  \right\rfloor$;
see Fig.~\ref{FGdtree}.
For a vertex $v$ in $T_{0}(x)$, 
the tree $T_{0}(\varphi(v))$ can be regarded as a subtree of $T_{0}(x)$.

\paragraph{Step 2:}
We next modify the definition of $D_{0}(x)$
so as to meet (C5) in addition to (C1)--(C4).
Note that 
$|D_{0}(x)| \geq \|x\|_{\infty}$
by (C1) and (C2).
To state two lemmas we introduce notations 
\[
\| x \|_\infty\sp{+} = \max(0, x_{1}, x_{2}, \ldots, x_{n} ),
\qquad
\| x \|_\infty\sp{-} = \max(0, -x_{1}, -x_{2}, \ldots, -x_{n} ).
\]
We have
$\| x \|_{\infty} = \max (\| x \|_{\infty}\sp{+}, \| x \|_{\infty}\sp{-})$.
Let us say that an integer vector $x$ is {\em critical} if 
$\| x \|_{\infty} \geq 3$, $\| x \|_{\infty}$ is odd, and 
$\| x \|_{\infty}\sp{+} =  \| x \|_{\infty}\sp{-}$.
A vertex $v$ of the tree $T_{0}(x)$ is also called 
{\em critical} if the associated vector $\varphi(v)$ 
is critical.
We modify the definition (\ref{D0def}) of $D_{0}(x)$ for critical $x$.

\begin{lemma}\label{LMtamProp3}
For $x \in \ZZ\sp{n}$
with $\| x \|_{\infty} \geq 2$,  we have
\[
\left\| \, \left\lceil x/2  \right\rceil \, \right\|_{\infty} +
\left\| \, \left\lfloor x/2  \right\rfloor \, \right\|_{\infty} 
 = \left\{ \begin{array}{ll}
  \| x \|_{\infty} +1  & (\mbox{$x$: \rm critical}) , \\
  \| x \|_{\infty}   & (\mbox{$x$: \rm non-critical}) . \\
 \end{array}\right.
\]
\end{lemma}
\begin{proof}
Let $\ell = \| x \|_{\infty}$.
If $x$ is critical, we have
$\left\| \, \left\lceil x/2 \right\rceil \, \right\|_{\infty} = 
\left\| \, \left\lfloor x/2 \right\rfloor \, \right\|_{\infty} 
= (\ell+1)/2$
and hence the claim holds.
Suppose that $x$ is not critical.
If $\ell$ is even, we have
$\left\| \, \left\lceil x/2 \right\rceil \, \right\|_{\infty} = 
\left\| \, \left\lfloor x/2 \right\rfloor \, \right\|_{\infty} 
=  \ell /2$
and hence the claim holds.
If $\ell$ is odd and
$\| x \|_{\infty}\sp{+} >  \| x \|_{\infty}\sp{-}$,
 we have
$\left\| \, \left\lceil x/2 \right\rceil \, \right\|_{\infty} = (\ell + 1)/2$
and
$\left\| \, \left\lfloor x/2 \right\rfloor \, \right\|_{\infty} = (\ell - 1)/2$,
and hence the claim holds.
If $\ell$ is odd and
$\| x \|_{\infty}\sp{+} <  \| x \|_{\infty}\sp{-}$,
 we have
$\left\| \, \left\lceil x/2 \right\rceil \, \right\|_{\infty} =  (\ell - 1)/2$
and
$\left\| \, \left\lfloor x/2 \right\rfloor \, \right\|_{\infty} = (\ell+1)/2$,
and hence the claim holds.
\end{proof}

\begin{lemma}\label{LMtamProp4}
Let $x \in \ZZ\sp{n}$.

\noindent
{\rm (1)}
If $\| x \|_{\infty}\sp{+} >  \| x \|_{\infty}\sp{-}$
and $\|x\|_{\infty} \geq 2$, then 
$y = \lceil x/2 \rceil$ satisfies
$\| y \|_{\infty}\sp{+} > \| y \|_{\infty}\sp{-}$.

\noindent
{\rm (2)}
If $\| x \|_{\infty}\sp{+} <  \| x \|_{\infty}\sp{-}$
and $\|x\|_{\infty} \geq 2$, then 
$z = \lfloor x/2 \rfloor$ satisfies
$\| z \|_{\infty}\sp{+} < \| z \|_{\infty}\sp{-}$.
\end{lemma}
\begin{proof}
(1) 
If $\| x \|_{\infty}\sp{+} =2k$ (even), then
$\| y \|_{\infty}\sp{+} =  k$ and $\| y \|_{\infty}\sp{-} \leq k -1$.
If $\| x \|_{\infty}\sp{+} =2k+1$ (odd), then
$\| y \|_{\infty}\sp{+} =  k+1$ and $\| y \|_{\infty}\sp{-} \leq k$.
(2) This can be shown similarly.
\end{proof}

Let $v$ be a critical vertex of $T_{0}(x)$ and $y = \varphi(v)$.
Let  $v\sp{\oplus}$ and $v\sp{\ominus}$, respectively, 
be the leftmost and rightmost leaves of the subtree below $v$.
Define 
$d\/\sp{\oplus}(v) = \varphi(v\sp{\oplus})$
and 
$d\/\sp{\ominus}(v) = \varphi(v\sp{\ominus})$.
For each vertex $w$ on the path 
between the left child of $v$ and the leaf $v\sp{\oplus}$
(inclusive), we have
$\| \varphi(w) \|_{\infty}\sp{+} > \| \varphi(w) \|_{\infty}\sp{-}$
by Lemma \ref{LMtamProp4} (1).
In particular,
$\| d\sp{\oplus}(v) \|_{\infty}\sp{+} > \| d\sp{\oplus}(v) \|_{\infty}\sp{-}$,
which implies
$d\sp{\oplus}(v) \in \{ 0, +1\}\sp{n}$
since 
$d\sp{\oplus}(v) \in \{-1, 0, +1\}\sp{n}$.
Similarly, we have
$d\sp{\ominus}(v) \in \{ 0, -1\}\sp{n}$
by Lemma \ref{LMtamProp4} (2).

\begin{example}\rm \label{EXtamEXdoplus}
The vector $x=(5,3,-3,-5)$ 
in Example~\ref{EXtamExample1}
is critical, since
$\| x \|_{\infty}=5 \geq 3$ is odd 
and $\| x \|_{\infty}\sp{+} =  \| x \|_{\infty}\sp{-}=5$.
Accordingly, in Fig.~\ref{FGdtree}, the root vertex $v_{1}$ is critical.
We have 
$v_{1}\sp{\oplus}=v_{8}$,
$v_{1}\sp{\ominus}=v_{11}$,
$d\/\sp{\oplus}(v_{1})=(1,1,0,0)$, and
$d\/\sp{\ominus}(v_{1})=(0,0,-1,-1)$.
\finbox
\end{example}

\begin{lemma}\label{LMtamD1new}
For distinct critical vertices $u$ and $v$ in $T_0(x)$, the leaves
$u\sp{\oplus}$, $u\sp{\ominus}$, $v\sp{\oplus}$, and $v\sp{\ominus}$
are all distinct.
\end{lemma}
\begin{proof}
For a critical vertex $v$, there is no critical vertex on the path 
between the left child of $v$ and the leaf $v\sp{\oplus}$
by Lemma \ref{LMtamProp4} (1),
and there is no critical vertex on the path 
between the right child of $v$ and the leaf $v\sp{\ominus}$
by Lemma \ref{LMtamProp4} (2).
The claim follows from this fact.
\end{proof}

With the above preparations we are now able to define $D_{1}(x)$ as follows%
\footnote{%%%%%%%%%%%%%%%
Recall that $D_{0}(x)$ is a multiset. 
The first term on the right-hand side of
 (\ref{D1xdef}) means that we decrease the multiplicities
of $d^{\oplus}(v)$ and $d^{\ominus}(v)$ for each critical vertex $v$.
}:  %%%%%%%%%%% foonote %%%%%%%%%%
\begin{align}
 D_{1}(x) =&  \big( D_0(x) \setminus 
 \{ d\/\sp{\oplus}(v), d\/\sp{\ominus}(v) \mid \mbox{$v \in T_0(x)$ is critical}  \} 
 \big)
 \nonumber\\
 & \cup
 \{ d\/\sp{\oplus}(v)+d\/\sp{\ominus}(v) \mid \mbox{$v \in T_0(x)$ is critical}  \}.
 \label{D1xdef}
\end{align}

\begin{example}\rm \label{EXtamExample2}
In Fig.~\ref{FGdtree} for $x=(5,3,-3,-5)$,
$v_{1}$ is the only critical vector with
$d\/\sp{\oplus}(v_{1})=(1,1,0,0)$ and
$d\/\sp{\ominus}(v_{1})=(0,0,-1,-1)$.
$D_{1}(x)$ is obtained from $D_{0}(x)$ in Example~\ref{EXtamExample1}
by deleting $(1,1,0,0)$ and $(0,0,-1,-1)$,
and adding their sum $(1,1,-1,-1)$.  Hence
$D_{1}((5,3,-3,-5)) 
\allowbreak 
= 
\allowbreak 
\{(1,0,0,-1), 
\allowbreak 
(1,1,-1,-1), 
\allowbreak 
(1,1,-1,-1), (1,0,0,-1), (1,1,-1,-1)\}$. 
\finbox
\end{example}

An alternative definition of $D_{1}(x)$,
which modifies the recursive definition of $D_{0}(x)$,
is also possible.
For a critical vector $y$ appearing in the recursive definition 
of $D_{0}(x)$,
we define 
$d\/\sp{\oplus}(y) = \varphi(v\sp{\oplus})$
and 
$d\/\sp{\ominus}(y) = \varphi(v\sp{\ominus})$
by choosing any vertex $v$ such that $y = \varphi(v)$.
This is well-defined since,
for two critical vertices $u$ and $v$ with 
$\varphi(u)= \varphi(v)$,
we have
$\varphi(u\sp{\oplus}) = \varphi(v\sp{\oplus})$
and
$\varphi(u\sp{\ominus}) = \varphi(v\sp{\ominus})$.
The recursive definition of $D_{1}(x)$ reads as follows:
\begin{equation} \label{D1def}
 D_{1}(x) = \left\{ \begin{array}{ll}
  \emptyset & (x = \veczero) , 
\\
  \{x\} & (\|x\|_{\infty} = 1), 
\\
   \{ \  \lceil x/2 \rceil, \  \lfloor x/2 \rfloor  \, \} 
    & (\|x\|_{\infty} = 2) ,
\\
D_{1}(\lceil x/2 \rceil) \cup D_{1}(\lfloor x/2 \rfloor) 
    & (\|x\|_{\infty} \geq 3, \  \mbox{$x$: non-critical}) ,
\\
 \multicolumn{2}{l}{
 (D_{1}(\lceil x/2 \rceil) \setminus \{d\/\sp{\oplus}(x)\}) 
  \cup
  (D_{1}(\lfloor x/2 \rfloor) \setminus \{d\/\sp{\ominus}(x)\}) 

  \cup \{d\/\sp{\oplus}(x) + d\/\sp{\ominus}(x) \} 
  }%%multicol
\\
   & (\|x\|_{\infty} \geq 3, \  \mbox{$x$: critical}). 
 \end{array}\right.
\end{equation}
In the last case, where
$\|x\|_{\infty} \geq 3$ and $x$ is critical,
the vectors $d\/\sp{\oplus}(x)$ and
$d\/\sp{\ominus}(x)$ are deleted from 
$D_{1}(\lceil x/2 \rceil)$ and 
$D_{1}(\lfloor x/2 \rfloor)$,
respectively, and 
their sum
$d\/\sp{\oplus}(x) + d\/\sp{\ominus}(x)
\in  \{-1,0, +1 \}\sp{n}$
is added instead.

\begin{lemma}\label{LMtamProp5}
$D_{1}(x)$ satisfies {\rm (C5)} in addition to {\rm (C1)--(C4)}.
\end{lemma}
\begin{proof}
First note that 
$d\/\sp{\oplus}(x) + d\/\sp{\ominus}(x) 
\in \{-1, 0, +1 \}\sp{n} \setminus \{ \veczero \}$.
Then it is easy to see that 
 (C1)--(C4) for $D_{0}(x)$ imply (C1)--(C4) for $D_{1}(x)$.

The condition (C5) can be shown as follows.
If $\| x \|_{\infty} \leq 2$, (C5) is obviously true by the definition.
Let $\ell = \| x \|_{\infty} \geq 3$.
To prove (C5) by induction, assume that 
(C5) is satisfied by $D_{1}(z)$
for any $z \in \ZZ\sp{n}$
with $\| z \|_{\infty} < \| x \|_{\infty}$.
Since 
$\|\, \lceil x/2 \rceil \,\|_{\infty} < \| x \|_{\infty}$ 
and
$\|\, \lfloor x/2 \rfloor \,\|_{\infty} < \| x \|_{\infty}$
by $\| x \|_{\infty} \geq 3$,
the induction hypothesis shows
\[
| D_{1}(\lceil x/2 \rceil) | = 
\|\, \lceil x/2 \rceil \,\|_{\infty}, \quad
| D_{1}(\lfloor x/2 \rfloor) | = 
\|\, \lfloor x/2\rfloor \,\|_{\infty} .
\]
If $x$ is not critical, we have
$\|\, \lceil x/2 \rceil \,\|_{\infty} + 
\|\, \lfloor x/2 \rfloor \,\|_{\infty} = \| x \|_{\infty}$
by Lemma~\ref{LMtamProp3}, and 
\[
 |D_{1}(x)| 
 = |D_{1}(\lceil x/2 \rceil)| + |D_{1}(\lfloor x/2 \rfloor)|
 = \|\, \lceil x/2 \rceil \,\|_{\infty} + 
   \|\, \lfloor x/2 \rfloor \,\|_{\infty}  = \| x \|_{\infty}.
\]
If $x$ is critical, we have
$\|\, \lceil x/2 \rceil \,\|_{\infty} + 
\|\, \lfloor x/2 \rfloor \,\|_{\infty} = \| x \|_{\infty} + 1$
by Lemma~\ref{LMtamProp3}, and 
\begin{align*}
 |D_{1}(x)| 
 & = (|D_{1}(\lceil x/2 \rceil)|-1)+ (|D_{1}(\lfloor x/2 \rfloor)|-1) + 1 
\\ & = \|\, \lceil x/2 \rceil \,\|_{\infty} + 
   \|\, \lfloor x/2 \rfloor \,\|_{\infty} - 1 = \| x \|_{\infty}.
\end{align*}
In either case, $D_{1}(x)$ satisfies (C5).
\end{proof}

\begin{lemma} \label{LMdecom-norm2}
We have
$\| d\sp{\circ} + d\sp{\bullet}  \|_{\infty} = 2$
for any two vectors $d\sp{\circ},  d\sp{\bullet} \in D_{1}(x)$.
\end{lemma}
\begin{proof}
If $\| d\sp{\circ} + d\sp{\bullet}  \|_{\infty} \leq 1$,
$x$ would be represented as a sum of 
$|D_{1}(x)|-1 = \| x \|_{\infty} -1$ 
vectors with $\ell_{\infty}$-norm $\leq 1$,
which is a contradiction.
\end{proof}
We note that Lemma \ref{LMdecom-norm2} can be adapted to 
any collection of vectors satisfying (C1)--(C5).

\paragraph{Step 3:}
Finally we modify $D_{1}(x)$
so as to meet (C6) in addition to (C1)--(C5).
Suppose that there are two vectors
$d\sp{\circ}, d\sp{\bullet} \in D_{1}(x)$ 
that are not comparable with each other, i.e.,
$d\sp{\circ} \not\leq d\sp{\bullet}$ and
$d\sp{\bullet} \not\leq d\sp{\circ}$.
Letting
\begin{equation}\label{vectwistdef}
 d\sp{\uparrow} = \left\lceil \frac{d\sp{\circ} + d\sp{\bullet}}{2} \right\rceil,\quad
 d\sp{\downarrow} = \left\lfloor \frac{d\sp{\circ} + d\sp{\bullet}}{2} \right\rfloor ,
\end{equation}
we modify $D_{1}(x)$ to 
\begin{equation}\label{D1twistdef}
 D_{1}'(x) = \big(  D_{1}(x) \setminus \{d\sp{\circ}, d\sp{\bullet}\} \big)
 \cup \{d\sp{\uparrow}, d\sp{\downarrow}\} ,
\end{equation}
where the incomparable pair 
$\{ d\sp{\circ}, d\sp{\bullet} \}$ 
is replaced by a comparable pair 
$\{ d\sp{\uparrow},d\sp{\downarrow} \}$
with
$d\sp{\uparrow} \geq  d\sp{\downarrow}$.

\begin{lemma} \label{LMdecom-twist1}
 $D_{1}'(x)$ satisfies {\rm (C1)}--{\rm (C5)}.
\end{lemma}
\begin{proof}
(C2) to (C5) for $D_{1}'(x)$ are easy to see.
As for (C1),
$d\sp{\uparrow}$ and $d\sp{\downarrow}$ are obviously
$\{-1, 0, +1 \}$-vectors.
They are nonzero since 
$\| d\sp{\circ} + d\sp{\bullet}  \|_{\infty} = 2$
by Lemma \ref{LMdecom-norm2} adapted to $D_{1}'(x)$.
\end{proof}

Repeated application of the modification (\ref{D1twistdef})
generates a sequence 
$D_{1}'(x), D_{1}''(x), \ldots$,
which ends up with 
$D_{2}(x) = \{d\sp{1},d\sp{2},\ldots,d\sp{\ell}\}$
satisfying the chain condition
$ d\sp{1} \leq d\sp{2} \leq \cdots \leq d\sp{\ell}$,
where $\ell = \|x\|_{\infty}$.
The recursive application of Lemma~\ref{LMdecom-twist1} shows that 
$D_{1}'(x), D_{1}''(x), \ldots, D_{2}(x)$ each satisfy (C1)--(C5).

\begin{example}\rm \label{EXtamExample3}
In Example~\ref{EXtamExample2},
$D_{1}((5,3,-3,-5))$ contains an incomparable pair of vectors
$d\sp{\circ} = (1,1,-1,-1)$ and
$d\sp{\bullet} = (1,0,0,-1)$. 
We have
$d\sp{\uparrow} = (1,1,0,-1)$ and $d\sp{\downarrow} = (1,0,-1,-1)$ 
in (\ref{vectwistdef}), and 
$D_{2}((5,3,-3,-5)) 
=  \{(1,0,-1,-1), (1,0,-1,-1), (1,1,-1,-1), (1,1,0,-1)$, $(1,1,0,-1)\}$.
\finbox
\end{example}

\paragraph{Step 4:}
For $x,y \in \ZZ\sp{n}$, consider
$D_{2}(y-x) = \{ d\sp{1},d\sp{2},\ldots,d\sp{m}\}$
with $d\sp{1} \leq d\sp{2} \leq \cdots \leq d\sp{m}$,
where $m = \| y - x \|_{\infty}$.
As is announced in (\ref{D2=1A1B}), we have the following relation
to the decomposition in (\ref{DICdecAkBksum}).

\begin{lemma}\label{LMtamD2=1A1B}
\quad $D_{2}(y-x) = \{ \bm{1}_{A_{k}} - \bm{1}_{B_{k}} \mid k=1,\ldots,m \}$.
\end{lemma}
\begin{proof}
First recall that 
$y - x = \sum_{k=1}\sp{m} (\bm{1}_{A_{k}} - \bm{1}_{B_{k}})$, \ 
$A_{1} \subseteq A_{2}  \subseteq \cdots \subseteq A_{m} \subseteq \suppp(x-y)$, \ 
$B_{m} \subseteq B_{m-1}  \subseteq \cdots \subseteq B_{1} \subseteq \suppm(x-y)$,
and
$\bm{1}_{A_{1}} - \bm{1}_{B_{1}} \leq \bm{1}_{A_{2}} - \bm{1}_{B_{2}} \leq
 \cdots \leq \bm{1}_{A_{m}} - \bm{1}_{B_{m}}$
in the decomposition in (\ref{DICdecAkBksum}).
For each vector $d^{k}$ in $D_{2}(y-x)$,
let
$d^{k}_{+}$, $d^{k}_{-}$, and $d^{k}_{0}$
denote its restriction  to (subvector on) 
to $\suppp(y-x)$, $\suppm(y-x)$, and 
$\{1,\ldots,n\} \setminus (\suppp(y-x) \cup \suppm(y-x))$,
respectively; see Fig.~\ref{FGvecdecAB}.
The properties (C1), (C3), and (C6) of $D_{2}(y-x)$ imply  
\[
 \veczero \leq d^{1}_{+} \leq d^{2}_{+} \leq \cdots \leq d^{m}_{+} \leq \vecone,
\quad
 -\vecone \leq d^{1}_{-} \leq d^{2}_{-} \leq \cdots \leq d^{m}_{-} \leq \veczero,
\quad
 d^{1}_{0} = d^{2}_{0} = \cdots = d^{m}_{0} = \veczero.
\]
Then we must have
$d^{k}_{+} = \vecone_{A_{k}}$ and $d^{k}_{-} = -\vecone_{B_{k}}$ for $k=1,\ldots,m$
by the properties (C2) and (C5). 
\end{proof}

\paragraph{Step 5:}
In Step 5, $S$
is assumed to be a discrete midpoint convex set in $\mathbb{Z}\sp{n}$
and $x, y \in S$.
For any subset $E$ of $D_{i}(y-x)$ with $i=0,1$, or $2$,
we consider the condition 
\begin{equation}\label{xdEinS}
 x + \sum\{d \mid d \in E\} \in S .
\end{equation}
Note that (\ref{xdEinS}) holds for $E=\emptyset$ by $x \in S$,
and for $E=D_{i}(y-x)$ by (C2) and $y \in S$.
It should  be clear that (\ref{xdEinS}) corresponds to (\ref{xJ1A1BinS}).

\begin{lemma}\label{LMtamProp6}
\quad

\noindent
{\rm (1)}
{\rm (\ref{xdEinS})} holds for any $E \subseteq D_{0}(y-x)$.

\noindent
{\rm (2)}
{\rm (\ref{xdEinS})} holds for any $E \subseteq D_{1}(y-x)$.
\end{lemma}
\begin{proof}
(1) This is trivially true if $\| y-x \|_{\infty} \leq 1$.
If $\| y-x \|_{\infty} = 2$, we have
$D_{0}(y-x) = \{\lceil (y-x)/2 \rceil, \lfloor (y-x)/2 \rfloor \}$.
Since $S$ is a discrete midpoint convex set,  we have
\begin{align*}
  x + \lceil (y-x)/2 \rceil &= \lceil (x+y)/2 \rceil \in S, 
\\ 
  x + \lfloor (y-x)/2 \rfloor &= \lfloor (x+y)/2 \rfloor \in S, 
\end{align*}
which show (\ref{xdEinS}) for $E$ with
 $\emptyset \not= E \not= D_{0}(y-x)$.

Let $\| y-x \|_{\infty} \geq 3$.
To prove the claim by induction, assume that (\ref{xdEinS}) is true 
for all $x', y' \in S$ with $\| y' - x' \|_{\infty} < \| y-x \|_{\infty}$
and for all $E \subseteq D_{0}(y'-x')$.
We now fix $E \subseteq D_{0}(y-x)$,
and define $x'' = \lfloor (x+y)/2 \rfloor $ and
$y'' = \lceil (x+y)/2 \rceil$, where 
$x'', y'' \in S$ by (\ref{SCdirintcnvset}).
Since
\begin{align*}
 \| y''- x \|_{\infty} &=  \|\,\lceil (y-x)/2 \rceil\,\|_{\infty} < \|y-x\|_{\infty},
\\
 \| y- x'' \|_{\infty} &=  \|\,\lceil (y-x)/2 \rceil\,\|_{\infty} < \|y-x\|_{\infty}
\end{align*}
by $\| y-x \|_{\infty} \geq 3$,
the induction hypothesis for $(x,y'')$ and $(x'',y)$ shows
\begin{align}
 & u = x + \sum\{ d \mid d \in E \cap D_{0}(\lceil (y-x)/2 \rceil) \} \in S, 
\label{xprimdec}
\\
 & v = x''
+ \sum\{ d \mid d \in E \cap D_{0}(\lceil (y-x)/2 \rceil) \} \in S,
\end{align}
where
$D_{0}(y''-x)=D_{0}(y - x'') = D_{0}(\lceil (y-x)/2 \rceil)$
is used.
Since 
\[
 \| v - u \|_{\infty}  =  \| x'' - x \|_{\infty}  = 
  \|\,\lfloor (y-x)/2 \rfloor\,\|_{\infty} < \|y-x\|_{\infty},
\]
we can also use the induction hypothesis for $(u, v)$ to obtain
\begin{equation} \label{udEinterDinS}
 w =  u + \sum\{ d \mid 
  d \in E \setminus D_{0}(\lceil (y-x)/2 \rceil)\} \in S ,
\end{equation}
where
$D_{0}(v-u) = D_{0}(\lfloor (y-x)/2 \rfloor)
\supseteq E \setminus D_{0}(\lceil (y-x)/2 \rceil)$ 
by
$D_{0}(\lceil (y-x)/2 \rceil) \cup D_{0}(\lfloor (y-x)/2 \rfloor)
 = D_{0}(y-x) \supseteq E$.
Substituting (\ref{xprimdec}) into (\ref{udEinterDinS})
we obtain
\[
  w = x + \sum\{d \mid d \in E\}  \in S ,
\]
which is nothing but (\ref{xdEinS}).

(2)
This follows from (1), since
each vector of $D_{1}(y-x)$ is a sum of some vectors of $D_{0}(y-x)$.
\end{proof}

Recall from Step 3 that 
$D_{2}(y-x)$ is constructed from $D_{1}(y-x)$ 
by repeated modification in (\ref{D1twistdef}),
which changes
an incomparable pair 
$\{ d\sp{\circ}, d\sp{\bullet} \}$ 
to a comparable pair 
$\{ d\sp{\uparrow},d\sp{\downarrow} \}$
defined by (\ref{vectwistdef}).

\begin{lemma}\label{LMdecom-twist2}
Let $D_{1}'$ be obtained from $D_{1}(y-x)$ as in {\rm (\ref{D1twistdef})}.
Then {\rm (\ref{xdEinS})} holds for any $E \subseteq D_{1}'$.
\end{lemma}
\begin{proof}
In accordance with (\ref{D1twistdef}),
we have
$ D_{1}' = \big(  D_{1}(y-x) \setminus \{d\sp{\circ}, d\sp{\bullet}\} \big)
 \cup \{d\sp{\uparrow}, d\sp{\downarrow}\}$.
We have four cases to consider.
\begin{enumerate}
\item
If 
$d\sp{\uparrow} \not\in E$ and $d\sp{\downarrow} \not\in E$,
we have
$E \subseteq D_{1}(y-x)$, and 
therefore
(\ref{xdEinS}) holds by Lemma~\ref{LMtamProp6}~(2).

\item
If
$d\sp{\uparrow}\in E$ and $d\sp{\downarrow} \in E$,
we have
$d\sp{\uparrow} + d\sp{\downarrow} = d\sp{\circ} + d\sp{\bullet}$
and  $(E\setminus \{d\sp{\uparrow}, d\sp{\downarrow}\})
 \cup\{d\sp{\circ}, d\sp{\bullet}\} \subseteq D_{1}(y-x)$.
Therefore,
\[
 x + \sum\{d \mid d \in E\} = x + \sum\{d \mid d \in E\setminus \{d\sp{\uparrow}, d\sp{\downarrow}\}\}
  + d\sp{\circ} + d\sp{\bullet} \in S
\]
by Lemma~\ref{LMtamProp6} (2).

\item
If
$d\sp{\uparrow} \in E$ and $d\sp{\downarrow} \not\in E$,
both points
\begin{align*}
 y &= 
x + \sum\{d \mid d \in E \setminus \{d\sp{\uparrow}\}\} + d\sp{\circ} + d\sp{\bullet},
 \\ z &= 
x + \sum\{d \mid d \in E \setminus \{d\sp{\uparrow}\}\}
\end{align*}
belong to $S$ by Lemma~\ref{LMtamProp6} (2).
We have $\lceil ( y + z)/2 \rceil \in S$
since $\| y - z  \|_{\infty} = \| d\sp{\circ} + d\sp{\bullet}  \|_{\infty} = 2$
(cf.~Lemma \ref{LMdecom-norm2})
and $S$ is a discrete midpoint convex set.
Therefore,
\[
 \lceil (y + z)/2 \rceil 
= x + \sum\{d \mid d \in E \setminus \{d\sp{\uparrow}\}\}
  + d\sp{\uparrow} \in S,
\]
which shows (\ref{xdEinS}).

\item
If $d\sp{\uparrow} \not\in E$ and $d\sp{\downarrow} \in E$,
we can show (\ref{xdEinS}) in a similar manner.
\end{enumerate}
\end{proof}

\begin{lemma}\label{LMdecom-twist2Muro} 
{\rm (\ref{xdEinS})} holds for any $E \subseteq D_{2}(y-x)$.
\end{lemma}

\begin{proof}
Repeated application of the modification (\ref{D1twistdef})
generates a sequence
$D_{1}', D_{1}'', \ldots$ that ends with $D_{2}(x-y)$.
The recursive application of Lemma~\ref{LMdecom-twist2} shows that 
{\rm (\ref{xdEinS})} holds for any 
$E \subseteq D_{1}'$, $E \subseteq D_{1}'', \ldots, E \subseteq D_{2}(y-x)$.
\end{proof}

Lemma~\ref{LMdecom-twist2Muro} with Lemma~\ref{LMtamD2=1A1B} 
establishes Theorem~\ref{THdirintcnvSetdec}.

\subsection*{Acknowledgements}
The authors thank Hiroshi Hirai for discussion
about {\rm L}-extendable functions
and Kazuya Tsurumi for Remark \ref{RMparaineqlocopt}.
This research was initiated at
the Trimester Program ``Combinatorial Optimization''
at Hausdorff Institute of Mathematics, 2015.
This work was supported by The Mitsubishi Foundation, 
CREST, JST, Grant Number JPMJCR14D2, Japan, and 
JSPS KAKENHI Grant Numbers JP26350430, JP26280004, JP24300003, JP16K00023, JP17K00037.

%%%JSPS/MEXT KAKENHI (Moriguchi=JP26350430,JP17K00037; Murota=JP26280004,Tamura=JP24300003,JP16K00023) 

\appendix
%%%%%%%%%%%%%% SSSSS %%%%%%%%%%%%%%%%%%%%%
\section{Equivalence of Integral Convexity and Global Weak DMC}%section
\label{SCwdmcIC}

We prove that the global weak discrete midpoint convexity 
of a function implies the integral convexity 
even without assuming the integral convexity of the effective domain.

\begin{theorem} \label{THicGlobWeakMPC}
For a function  $f: \mathbb{Z}^{n} \to \mathbb{R} \cup \{ +\infty  \}$,
the following properties are equivalent:

{\rm (a)}
$f$ is integrally convex.

{\rm (b)}
For every $x, y \in \dom f$ with $\| x - y \|_{\infty} \geq 2$ 
we have \ 
\begin{equation}  \label{intcnvconddist234}
 f(x) + f(y) \geq 2 \tilde{f}\, \bigg(\frac{x + y}{2} \bigg).
\end{equation}
\end{theorem}

\begin{proof}
Obviously, (a) implies (b). 
To prove ``(b) $\Rightarrow$ (a)'' by contradiction,
assume that there exist
$x \in \overline{\dom f}$
and $y\sp{1},\ldots, y\sp{m} \in \dom f$ such that
\begin{equation} \label{xyifxfyi-2}
 x =  \sum_{i=1}\sp{m} \lambda_{i} y\sp{i} ,  
\qquad
\tilde{f}(x) > \sum_{i=1}\sp{m} \lambda_{i} f(y\sp{i}) ,
\end{equation}
where
$\sum_{i=1}\sp{m}  \lambda_{i} = 1$ and 
$\lambda_{i} > 0 \ (i=1,\ldots, m)$.
For each $j=n,n-1,\ldots, 1$, we look at the 
$j$-th component of the generating points $y\sp{i}$.

Let $j=n$ and define
\begin{equation} \label{alphabetaIminImax}
\alpha_{n} = \min_{i} y\sp{i}_{n},
\quad
\beta_{n} = \max_{i} y\sp{i}_{n},
\quad
I_{\min} = \{ i \mid y\sp{i}_{n} = \alpha_{n} \},
\quad
I_{\max} = \{ i \mid y\sp{i}_{n} = \beta_{n} \} .
\end{equation}
If $\beta_{n} - \alpha_{n} \leq 1$, we are done with $j=n$.
Suppose that 
$\beta_{n} - \alpha_{n} \geq 2$.
By translation and reversal of the $n$-th coordinate,
we may assume $0 \leq  x_{n} \leq  1$,
$\alpha_{n} \leq 0$, and
$\beta_{n} \geq 2$.
By renumbering the generators we may assume
$1 \in  I_{\min}$
and
$2 \in  I_{\max}$,
i.e., $y\sp{1}_{n} = \alpha_{n}$ and $y\sp{2}_{n} = \beta_{n}$.
We have
$\| y\sp{1} - y\sp{2} \|_{\infty} \geq 2$. 

By (\ref{intcnvconddist234}) for $(y\sp{1}, y\sp{2})$ 
and the definition of $\tilde{f}$ we have
\begin{equation} \label{fy1y2zk-2}
 f(y\sp{1}) + f(y\sp{2})
 \geq 2 \tilde{f}\, \bigg(\frac{y\sp{1} + y\sp{2}}{2} \bigg) 
 = 2 \sum_{k=1}\sp{l} \mu_{k} f(z\sp{k}) ,
\end{equation}
where
\begin{equation} \label{y1y2zk-2}
\frac{y\sp{1} + y\sp{2}}{2} =  \sum_{k=1}\sp{l}  \mu_{k} z\sp{k},
\qquad
z\sp{k} \in N \bigg( \frac{y\sp{1} + y\sp{2}}{2} \bigg) \cap \dom f
\quad
(k=1,\ldots, l)
\end{equation}
with 
$\mu_{k} > 0$ \  $(k=1,\ldots, l)$ and
$\sum_{k=1}\sp{l}  \mu_{k} = 1$.
The inequality (\ref{fy1y2zk-2}) implies,
with notation $\lambda = \min(\lambda_{1}, \lambda_{2})$,
that 
\[
 \lambda_{1} f(y\sp{1}) + \lambda_{2} f(y\sp{2}) 
\geq
 (\lambda_{1} - \lambda ) f(y\sp{1}) + (\lambda_{2}-\lambda) f(y\sp{2})
 + 2 \lambda \sum_{k=1}\sp{l} \mu_{k} f(z\sp{k}) .
\]
Hence
\begin{align*}
 \sum_{i=1}\sp{m} \lambda_{i} f(y\sp{i}) 
 &=
 \left( \lambda_{1} f(y\sp{1}) + \lambda_{2} f(y\sp{2}) \right)
+  \sum_{i=3}\sp{m} \lambda_{i} f(y\sp{i}) 
\\
 &\geq
 (\lambda_{1} - \lambda ) f(y\sp{1}) + (\lambda_{2}-\lambda) f(y\sp{2})
 + 2 \lambda \sum_{k=1}\sp{l} \mu_{k} f(z\sp{k}) 
+  \sum_{i=3}\sp{m} \lambda_{i} f(y\sp{i}) .
\end{align*}
Since  
\[
 x =  (\lambda_{1} - \lambda ) y\sp{1} + (\lambda_{2}-\lambda) y\sp{2}
 + 2 \lambda \sum_{k=1}\sp{l} \mu_{k} z\sp{k}
+  \sum_{i=3}\sp{m} \lambda_{i} y\sp{i} ,
\]
we have obtained another representation of the form 
(\ref{xyifxfyi-2}).

With reference to this new representation
we define
$\hat \alpha_{n}$, $\hat \beta_{n}$,
$\hat{I}_{\min}$, and 
$\hat{I}_{\max}$, 
as in (\ref{alphabetaIminImax}).
Since $\beta_{n} - \alpha_{n} \geq 2$, we have 
\[
 \alpha_{n} + 1 \leq (y_{n}\sp{1} + y_{n}\sp{2})/2  \leq \beta_{n} - 1  ,
\]
which implies $\alpha_{n} + 1 \leq z_{n}\sp{k} \leq \beta_{n} - 1$ for all $k$.
Hence, 
$\alpha_{n} \leq \hat \alpha_{n}$ and
$\hat \beta_{n} \leq \beta_{n}$.
Moreover, if 
$(\hat \alpha_{n}, \hat \beta_{n})=(\alpha_{n},\beta_{n})$,
then 
$|\hat{I}_{\min}| +  |\hat{I}_{\max}|  \leq |I_{\min}| +  |I_{\max}| - 1$.
Therefore, by repeating the above process with $j=n$, 
 we eventually arrive at 
a representation of the form of (\ref{xyifxfyi-2}) 
with $\beta_{n} - \alpha_{n} \leq 1$.

We next apply the above procedure for the $(n-1)$-st component.
What is crucial here is that 
the condition $\beta_{n} - \alpha_{n} \leq 1$
is maintained in the modification of the generators
by (\ref{fy1y2zk-2}) for the $(n-1)$-st component.
Indeed, 
for each $k$, the inequality
$\alpha_{n} \leq z\sp{k}_{n} \leq \beta_{n}$ 
follows from 
$\alpha_{n} \leq (y\sp{1}_{n} +y\sp{2}_{n})/2 \leq \beta_{n}$
and
$z\sp{k} \in N( (y\sp{1} + y\sp{2})/2 )$.
Therefore, we can obtain 
a representation of the form of (\ref{xyifxfyi-2}) 
with 
$\beta_{n} - \alpha_{n} \leq 1$ and $\beta_{n-1} - \alpha_{n-1} \leq 1$,
where
$\alpha_{n-1} = \min_{i} y\sp{i}_{n-1}$ and $\beta_{n-1} = \max_{i} y\sp{i}_{n-1}$.

Then we continue the above process for $j=n-2,n-3, \ldots,1$,
to finally obtain a representation of the form of (\ref{xyifxfyi-2})
with $|y\sp{i}_{j} - y\sp{i'}_{j}| \leq 1$
for all $i, i'$ and $j=1,2,\dots,n$.
This contradicts the definition of $\tilde{f}$.
\end{proof}

\end{document}